\documentclass[12pt]{amsart}
\usepackage{amsmath,amssymb,verbatim,amscd,bm,color,yhmath,bm, enumerate}
\usepackage[normalem]{ulem}
\usepackage[driverfallback=dvipdfm]{hyperref}

\usepackage[active]{srcltx} 

\newtheorem{lem}{Lemma}[section]
\newtheorem{thm}[lem]{Theorem}
\newtheorem{prop}[lem]{Proposition}
\newtheorem{cor}[lem]{Corollary}
\newtheorem{remark}[lem]{Remark}

\theoremstyle{definition}
\newtheorem{defn}{Definition}[section]
\newtheorem*{ex}{Example}
\newtheorem{rem}{Remark}
\newtheorem*{nota}{Notation}

\numberwithin{equation}{section}

 \newcommand{\T}{\mathbb T}
 \newcommand{\N}{\mathbb N}

 \newcommand{\la}{\lambda}
 \newcommand{\D}{\displaystyle}

 \newcommand{\q}{\quad}
 \newcommand{\qq}{\qquad}

 \newcommand{\Ch}[1]{\operatorname{Ch}(#1)}

 \newcommand{\ov}{\overline}

 \newcommand{\fr}{\frac}

 \newcommand{\set}[1]{\{ #1 \}}

 \newcommand{\cha}{\operatorname{Ch}(A)}
 \newcommand{\chb}{\Ch{B}}

 \newcommand{\al}{\alpha}
 \newcommand{\be}{\beta}
 
 \renewcommand{\Re}{{\rm Re}\,}

 \renewcommand{\d}{\delta}
 \renewcommand{\set}[1]{\{ #1 \}}

 \newcommand{\Vinf}[1]{\Vert #1 \Vert}

 \newcommand{\e}{\varepsilon}

 \newcommand{\ran}{\operatorname{Ran}}
 \newcommand{\rpi}{\ran_\pi}
 
 \DeclareSymbolFont{largesymbol}{OMX}{yhex}{m}{n}
 \DeclareMathAccent{\Widehat}{\mathord}{largesymbol}{"62}

 \DeclareSymbolFont{largesymbol}{OMX}{yhex}{m}{n}

 \newcommand{\SA}{S(A)}
 \newcommand{\SB}{S(B)}
 \newcommand{\del}{\Delta}
 \newcommand{\deli}{\Delta^{-1}}
 \newcommand{\alx}{\al_x}

 \newcommand{\KP}{K_+}
 \newcommand{\KM}{K_-}
 \newcommand{\LP}{L_+}
 \newcommand{\LM}{L_-}

\usepackage{ulem}
\usepackage{marginnote}
\usepackage{geometry}

\begin{document}



\title[Tingley's problem for function algebras]
{Exploring new solutions to Tingley's problem for function algebras}

\author[M. Cueto-Avellaneda]{Mar{\'i}a Cueto-Avellaneda}
\address[M. Cueto-Avellaneda]{
School of Mathematics, Statistics and Actuarial Science, University of Kent, Canterbury, Kent CT2 7NX, UK}
\email{emecueto@gmail.com}

\author[D. Hirota]{Daisuke Hirota}
\address[D. Hirota]{Graduate School of Science and Technology, Niigata University, Niigata 950-2181, Japan}
\email{hirota@m.sc.niigata-u.ac.jp}

\author[T. Miura]{Takeshi Miura}
\address[T. Miura]{Department of Mathematics, Faculty of Science, Niigata University, Niigata 950-2181, Japan}
\email{miura@math.sc.niigata-u.ac.jp}

\author[A.M. Peralta]{Antonio M. Peralta}
\address[A.M. Peralta]{Instituto de Matem{\'a}ticas de la Universidad de Granada (IMAG), Departamento de An{\'a}lisis Matem{\'a}tico, Facultad de
Ciencias, Universidad de Granada, 18071 Granada, Spain.}
\email{aperalta@ugr.es}


\begin{abstract}
In this note we present two new positive answers to Tingley's problem in certain subspaces of function algebras. In the first result we prove that every surjective isometry between the unit spheres, $S(A)$ and $S(B)$, of two uniformly closed function algebras $A$ and $B$ on locally compact
Hausdorff spaces can be extended to a surjective real linear isometry from $A$ onto $B$. In a second goal we study surjective isometries between the unit spheres of two abelian JB$^*$-triples represented as spaces of continuous functions of the form $$C^{\mathbb{T}}_0 (X) := \{ a \in C_0(X) : a (\lambda t) = \lambda a(t) \hbox{ for every } (\lambda, t) \in \mathbb{T}\times X\},$$ where $X$ is a (locally compact Hausdorff) principal $\mathbb{T}$-bundle. We establish
that every surjective isometry $\Delta: S(C_0^{\mathbb{T}}(X))\to S(C_0^{\mathbb{T}}(Y))$ admits an extension to a surjective real linear isometry between these two abelian JB$^*$-triples.
\end{abstract}

\maketitle


\section{Preliminaries}

The problem of extending a surjective isometry between the unit spheres of two Banach spaces --named \emph{Tingley's problem} after the contribution of D. Tingley in \cite{Ting1987}-- is nowadays a trending topic in functional analysis (see a representative sample in the references \cite{
CabSan19, CuePer18, CuePer19, FerJorPer2018, FerPe17c, FerPe17b, FerPe17d, FerPe18Adv, Mori2017, MoriOza2018, Pe2020FILOMAT, CamposJGarciaPacheco21} and the surveys \cite{YangZhao2014, Pe2018}). This isometric extension problem remains open for Banach spaces of dimension bigger than or equal to 3 though. In fact, it has not been until recently that a complete positive solution for 2-dimensional Banach spaces was obtained by T. Banakh in \cite{Banakh21SolutionTingley2dim}, a result culminating a tour-de-force by several  researchers (cf. \cite{Banakh21Smoothtwodim,BanakhCabello21,CabSan19}).\smallskip

In the last years, a growing interest on Tingley's problem for surjective isometries between the unit spheres of certain function algebras has attracted different specialists to approach this problem. The pioneering paper by R. Wang \cite{Wang} inspired many subsequent results. O. Hatori, S. Oi and R.S. Togashi proved that each surjective isometry between the unit spheres of two uniform algebras can be always extended to a surjective real linear isometry between the uniform algebras (cf. \cite{HatOiTog}). We recall that a uniform algebra is a closed subalgebra of $C(K)$ which contains constants and separates the points of $K$, where the latter is a compact Hausdorff space and $C(K)$ denotes the Banach algebra of all complex-valued continuous functions on $K$. This conclusion was improved by O. Hatori by showing that each uniform algebra $\mathcal{A}$ satisfies the complex Mazur-Ulam property, that is, every surjective isometry from its unit sphere onto the unit sphere of another Banach space $E$ extends to a surjective real linear isometry from $\mathcal{A}$ onto $E$ (see \cite[Theorem 4.5]{Hat2021}).\smallskip

This note is aimed to present our recent advances on Tingley's problem for some Banach spaces which are representable as certain function spaces. More concretely, in section \ref{sec:function algebras} we study Tingley's problem in the case of a surjective isometry between the unit spheres of two uniformly closed function algebras. Note that uniformly closed function algebras constitute a strictly wider class than that given by uniform algebras. Indeed, we begin with a locally compact Hausdorff space $X$. A {\it uniformly closed function algebra} $A$ on $X$ is a uniformly closed and strongly separating (i.e. for each $x\in X$ and $y,z\in X$ with $y\neq z$, there exist $f,g\in A$ such that $f(x)\neq 0$ and $g(y)\neq g(z)$) subalgebra of the algebra, $C_0(X),$ of all continuous complex-valued functions on $X$ vanishing at infinity. We can obviously regard $A$ as a subalgebra of $C_0(X\cup\{\infty\})$, where $X\cup\{\infty\}$ denotes the one-point compactification of $X$. However, it is worth observing that, under such an identification, $A$ never contains the constant functions, so, it is not a uniform algebra.\smallskip

The main conclusion of section \ref{sec:function algebras} proves that surjective isometry $\del\colon\SA\to\SB$ between the unit spheres of two uniformly closed function algebras $A$ and $B$, extends to a surjective real linear isometry $T\colon A\to B$ (see Theorem \ref{thm1}). Our arguments are based on an appropriate use of the Choquet boundary of each uniformly closed function algebra, the existence of Urysohn's lemma type properties for this Choquet boundary (as in \cite{RaoRoy2005, RaoRoy2005II, FleJam, miu1}) and a good description of the elements in the image of $\Delta$ at points in the Choquet boundary. The proof of the already mentioned result by Hatori, Oi and Togashi in \cite{HatOiTog} is inspired by some of Wang's original tools in \cite{Wang}. In this manuscript we apply similar techniques, however, the arguments here provide a different point of view, and are not mere extensions to the case of uniformly closed function algebras as non-unital versions of uniform algebras.\smallskip

The third section of this note is focused on the study of Tingley's problem for a surjective isometry between the unit spheres of two abelian JB$^*$-triples. As it is well-known, and explained in section \ref{sec: abelian JBstar triples}, JB$^*$-triples are precisely those complex Banach spaces whose open unit ball is a bounded symmetric domain (cf. \cite{Ka83}). A JBW$^*$-triple is a JB$^*$-triple which is also a dual Banach space. It has been recently shown (cf. \cite{BeCuFerPe2018,KalPe2019}) that every surjective isometry from the unit sphere of a JBW$^*$-triple onto the unit sphere of another Banach space extends to a surjective real linear isometry between the spaces. Few or nothing is known for general JB$^*$-triples. The elements in the subclass of abelian JB$^*$-triples can be identified, thanks to a Gelfand representation theory, with subspaces of continuous functions. Indeed, let $X$ be a principal $\mathbb{T}$-bundle (i.e. a subset of a Hausdorff locally convex complex space such that $0 \notin X$, $X \cup \{0\}$ is compact, and $\mathbb{T} X \subseteq X$, where $\mathbb{T}= S(\mathbb{C})$). When $X$ is regarded as a locally compact Hausdorff space, the closed subspace of $C_0(X)$ defined by $$C^{\mathbb{T}}_0 (X) := \{ a \in C_0(X) : a (\lambda t) = \lambda a(t) \hbox{ for every } (\lambda, t) \in \mathbb{T}\times X\},$$ is not, in general, a subalgebra of $C_0(X)$ but it is closed for the triple product $\{a,b,c\} = a \overline{b} c$ ($\,a,b,c\in C^{\mathbb{T}}_0 (X) $). The Gelfand representation theory affirms that each abelian JB$^*$-triple is isometrically isomorphic to some  $C^{\mathbb{T}}_0 (X)$ for a suitable principal $\mathbb{T}$-bundle $X$ (see \cite[Corollary 1.11]{Ka83}).  These spaces are also related to Lindenstrauss spaces (cf. \cite[Theorem 12]{Ol74}).\smallskip

The main conclusion in section \ref{sec: abelian JBstar triples} establishes that each surjective isometry
$\Delta : S(C^{\mathbb{T}}_0 (X)) \to S(C^{\mathbb{T}}_0 (Y)),$ with $X$ and $Y$ being two principal $\mathbb{T}$-bundles, admits an extension to a surjective real linear isometry $T : C^{\mathbb{T}}_0 (X)\to C^{\mathbb{T}}_0 (Y)$ (see Theorem \ref{t Tingleys problem for commutative JBstar triples}). This statement is complemented with Lemma \ref{l form of surjective real linear isometries between abelian JBstar triples} where it is shown that for each surjective real linear isometry
$T : C^{\mathbb{T}}_0 (X) \to C^{\mathbb{T}}_0 (Y)$ there exist a $\mathbb{T}$-invariant clopen subset $D\subseteq X$ and a $\mathbb{T}$-equivariant homeomorphism $\phi : Y\to X$ satisfying $$T(a) (s) = a(\phi(s)), \hbox{ for all } a\in C^{\mathbb{T}}_0 (X) \hbox{ and for all } s\in \phi^{-1}(D),$$ and $$T(a) (s) = \overline{a(\phi(s))}, \hbox{ for all } a\in C^{\mathbb{T}}_0 (X)\hbox{ and for all } s\in \phi^{-1}(X\backslash D).$$

Tingley's problem for surjective isometries between the unit spheres of function spaces deserves its own attention, and a self-contained treatment. In a forthcoming paper we shall explore the Mazur--Ulam property for the function spaces studied in this note.

\section{Tingley's problem for uniformly closed function algebras}\label{sec:function algebras}

Let $X$ be a locally compact Hausdorff space. Along this note we denote by $C_0(X)$ the set of all continuous complex-valued functions $f$ on $X$, which vanish at
infinity in the usual sense: for each $\e>0$ the set $\set{x\in X:|f(x)|\geq\e}$ is a compact subset of
$X$. Then $C_0(X)$ is a commutative Banach algebra under pointwise operations and the supremum norm $\Vinf{f}=\sup_{x\in X}|f(x)|$ ($f\in C_0(X)$).
A subset $B$ of $C_0(X)$ is said to be {\it strongly separating}, if for each $x\in X$ and
$y,z\in X$ with $y\neq z$, there exist $f,g\in B$ such that $f(x)\neq0$ and $g(y)\neq g(z)$.
A {\it uniformly closed function algebra} $A$ on $X$ is a uniformly closed and strongly
separating subalgebra of $C_0(X)$.\smallskip

For each function $f\in A$ the symbol $\ran(f)$ will stand for the range of $f$. We set  $\rpi(f)=\set{z\in\ran(f):|z|=\Vinf{f}}$ ($f\in A$).
A {\it peaking function} $g$ for $A$ is a function of $A$ with
$\rpi(g)=\set{1}$; that is, if $g\in A$ satisfies $\Vinf{g}=1$ and $|g(x)|=1$ for $x\in X$,
then $g(x)=1$. A compact subset $\mathcal{P}\subset X$ is called a \emph{peak set} of $A$ if there exists a peaking function $f \in A$ for which $\mathcal{P}= \{ x \in X : f(x) = 1\}$. A subset which coincides with an intersection of a family of peak sets of $A$ is called a \emph{weak peak set} of $A$. A \emph{peak point} (respectively, a \emph{weak peak point}) of $A$ is an element $x \in X$  satisfying that $\{x\}$ is a peak set (respectively, a \emph{weak peak set}) of $A$. The {\it Choquet boundary} or  the \emph{strong boundary} for $A$,  denoted by $\cha$, is the set of all weak peak points of $A$.  It is shown in \cite[Theorem 2.1]{RaoRoy2005II} (see also \cite{RaoRoy2005}) that $\cha$ is precisely the set of all $x \in X$ such that the evaluation functional at the point $x,$  $\delta_x$, is an extreme point of the unit ball of the dual space of $A$ (cf. \cite[Definition 2.3.7]{FleJam}). It is well known that $\cha$ is indeed a boundary (norming set) for $A$; furthermore, $\cha$ satisfies the following properties (see, for example, \cite[Propositions~2.2 and 2.3]{miu1}):\label{properties of the Choquet boundary}
\begin{enumerate}[$(1)$]
\item For each $f\in A$ there exists $x\in\cha$ such that $|f(x)|=\Vinf{f}$;
\item
For each $\e>0$, $x\in\cha$ and each open subset $O$ in $X$ with $x\in O$ there exists a peaking function $u\in A$ such that $u(x)=1$ and $|u|<\e$ on $X\setminus O$.
\end{enumerate}

The following is the main result of this section.

\begin{thm}\label{thm1}
Let $\SA$ and $\SB$ be unit spheres of two uniformly closed function algebras $A$ and $B$, respectively.
If $\del\colon\SA\to\SB$ is a surjective isometry, then
there exists a surjective, real linear isometry $T\colon A\to B$ such that $T=\del$ on $\SA$.
\end{thm}

\begin{rem}
Let $T\colon A\to B$ be a surjective real linear isometry.
In \cite[Theorem 1.1]{miu1}, such an isometry $T$ was characterized as a weighted composition operator, more concretely, there exist a continuous function $\kappa : \chb \to \mathbb{T}= \{\lambda \in \mathbb{C} : |\lambda| = 1\}$, a (possibly empty) clopen subset $K$ of $\chb$, and a homeomorphism $\varphi: \chb \to \cha$ such that
$$T(f)(y) =
\begin{cases}
	\kappa(y) f(\varphi(y)), &    \hbox{ for } y\in K,\\[2mm]
	\kappa(y) \overline{f(\varphi(y))}, &    \hbox{ for } y\in \chb\backslash K,    \end{cases}%
$$ for all $f\in A$.
\end{rem}

\begin{nota} Under the previous assumptions, for each $f\in\SA$, we write $|f|^{-1}(1)$ for the set $\set{x\in X:|f(x)|=1}$, and we set
$$ M_f=|f|^{-1}(1)\cap\cha.$$

For $x\in\cha$, we denote by $P_x$ the set of all peaking functions $f$ for $A$
with $f(x)=1$. Define $\la P_x=\set{\la f:f\in P_x}$ for each $\la\in\T$.
In the same way, we define $Q_y$ the set of all peaking functions $u$
for $B$ with $u(y)=1$.
Here we note that
$$
M_f=\set{z\in\cha:|f(z)|=1}=\set{z\in\cha:f(z)=\la}
$$
for all $f\in\la P_x$ and $\la\in\T$.\smallskip

For each $\la\in\T$ and $x\in\cha$, we define
$$
\la V_x=\set{f\in\SA:f(x)=\la}.
$$
We see that $\la P_x\subset\la V_x$.
In the same way, we define $\mu W_y=\set{u\in\SB:u(y)=\mu}$
for each $\mu\in\T$ and $y\in\chb$.

\end{nota}

\begin{lem}\label{lem1}
Let $f,g\in\SA$ and $x_0\in M_f$. If $f(x_0)\neq g(x_0)$, then
there exists $h\in\SA$ such that $\Vinf{f-h}=2>\Vinf{g-h}$.
\end{lem}

\begin{proof}
Note first that $|f(x_0)|=1$, since $x_0\in M_f$.
Set $2\d=|f(x_0)-g(x_0)|$, and then $\d>0$.
Define the open neighborhood $O$ of $x_0$
by $O=\set{x\in X:|g(x)-g(x_0)|<\d}$. Since $x_0\in M_f\subset\cha$,
there exists $u\in P_{x_0}$ such that $|u|<2^{-1}$ on $X\setminus O$.
We set $h=-f(x_0)u\in\SA$. We have
$$
2=|2f(x_0)|=|f(x_0)-h(x_0)|\leq\Vinf{f-h}\leq2,
$$
and thus $\Vinf{f-h}=2$.\smallskip

Take an arbitrary $x\in X$. We shall prove that $|g(x)-h(x)|<2$.
If $x\in O$, then we observe that $|g(x)-h(x)|<2$. Indeed, if $|g(x)-h(x)|=2$, then $g(x)=-h(x)$ and $|h(x)|=1$, since $g,h\in\SA$.
This implies that $|u(x)|=|h(x)|=1$.
Since $u$ is a peaking function for $A$, we obtain $u(x)=1$, and hence
$g(x)=-h(x)=f(x_0)$. Since $x\in O$, we get
$2\d=|f(x_0)-g(x_0)|=|g(x)-g(x_0)|<\d$, a contradiction.
We have proved that $|g(x)-h(x)|<2$ for all $x\in O$.
Suppose now that $x\in X\setminus O$. Then $|u(x)|<2^{-1}$.
It follows that
$$
|g(x)-h(x)|\leq|g(x)|+|f(x_0)u(x)|\leq1+\fr{1}{2}<2.
$$
Hence, $|g(x)-h(x)|<2$, and consequently, $\Vinf{g-h}<2$.
\end{proof}

In the rest of this section, we assume that $A$ and $B$ are uniformly closed function algebras on locally compact Hausdorff spaces $X$ and $Y$, respectively, and that
$\del\colon\SA\to\SB$ is a surjective isometry with respect to the supremum norms.

\begin{lem}\label{lem2}
Let $f,g\in\SA$. If $f=g$ on $M_f$, then $\del(f)=\del(g)$ on $M_{\del(f)}$.
\end{lem}

\begin{proof} Arguing by contradiction we suppose the existence of $y_0\in M_{\del(f)}$ such that $\del(f)(y_0)\neq\del(g)(y_0)$.
Applying Lemma~\ref{lem1} to $\del(f),\del(g)\in\SB$ and $y_0\in M_{\del(f)}$,
we can choose $h\in\SA$ so that $\Vinf{\del(f)-\del(h)}=2>\Vinf{\del(g)-\del(h)}$,
here we have used that $\del$ is surjective.
Since $\del$ is an isometry, we have
$\Vinf{f-h}=2>\Vinf{g-h}$.
Recall that $\cha$ is a boundary for $A$, and thus there exists
$x_0\in\cha$ with $|f(x_0)-h(x_0)|=2$, and by the other condition $|g(x_0)-h(x_0)|<2$.
Since $f,h\in\SA$, we get $|f(x_0)|=1$, which implies
$x_0\in M_f$.
Consequently, $f(x_0)\neq g(x_0)$ for $x_0\in M_f,$ which is impossible.
\end{proof}

\begin{lem}\label{lem3}
Let $x\in\cha$, $\la\in\T$ and $n\in\N$.
If $f_j\in\la P_x$ for each $j\in\N$ with $1\leq j\leq n$,
then $g=n^{-1}\sum_{j=1}^nf_j\in A$ satisfies $g\in\la P_x$ with
$M_g\subset\cap_{j=1}^nM_{f_j}$.
\end{lem}

\begin{proof}
Since $f_j\in\la P_x$ for $j=1,2,\ldots,n$,
then $f_j(x)=\la$ and $\Vinf{f_j}=1$ for all $j\in\set{1,2,\ldots,n}$. We have
$$
n=|n\la|=\left|\sum_{j=1}^nf_j(x)\right|
\leq\sum_{j=1}^n|f_j(x)|\leq\sum_{j=1}^n\Vinf{f_j}=n.
$$
Hence, $g=n^{-1}\sum_{j=1}^nf_j\in A$ satisfies $\ov{\la}g(x)=1=\Vinf{g}$.

We shall prove that $g\in\la P_x$.
Suppose that $|\ov{\la}g(x')|=1$ for $x'\in X$, and
then $|\sum_{j=1}^nf_j(x')|=n$. Since $|f_j(x')|\leq 1$, it follows that $|\ov{\la} f_j(x')|=|f_j(x')|=1$ for all $j\in\set{1,2,\ldots,n}$, which implies that $\ov{\la}f_j(x')=1$ because $\ov{\la}f_j\in P_x$,
and thus $\ov{\la}g(x')=1$. This shows that $\ov{\la}g\in P_x$ (consequently, $g\in\la P_x$) and $M_g\subset\cap_{j=1}^nM_{f_j}$. 

\end{proof}

The characterization of compactness in terms of the finite intersection property is employed in our next result.

\begin{lem}\label{lem4}
The intersection $\displaystyle \bigcap_{u\in\del(\la P_x)}|u|^{-1}(1)$ is non-empty for all $\la$ in $\T$ and  $x$ in $\cha$.
\end{lem}

\begin{proof}
First, we note that $|u|^{-1}(1)$ is a compact subset of $Y$ for each
$u\in\del(\la P_x)$. So, by the characterization of compactness in terms of the finite intersection property, it is enough to show that
$\cap_{j=1}^n|u_j|^{-1}(1)\neq\emptyset$ for each $n\in\N$ and
$u_j\in\del(\la P_x)$ with $j=1,2,\ldots,n$.\smallskip

Let $n\in\N$ and $u_j\in\del(\la P_x)$ for $j=1,2,\ldots,n$.
Choose $f_j\in\la P_x$ so that $u_j=\del(f_j)$, and set
$g=n^{-1}\sum_{j=1}^nf_j\in A$.
We see that $g\in\la P_x$ with $M_g\subset\cap_{j=1}^nM_{f_j}$
by Lemma~\ref{lem3}.
We shall prove that $M_g=\cap_{j=1}^nM_{f_j}$.
Here, we recall that
\begin{equation}\label{lem4.1}
M_f=\set{z\in\cha:|f(z)|=1}=\set{z\in\cha:f(z)=\la}
\end{equation}
for all $f\in\la P_x$.
Let $x_0\in\cap_{j=1}^nM_{f_j}$.
Since $f_j\in\la P_x$, we have $f_j(x_0)=\la$ for all $j\in\set{1,2,\ldots,n}$
by \eqref{lem4.1}.
It follows that $g(x_0)=n^{-1}\sum_{j=1}^nf_j(x_0)=\la$.
Since $g\in\la P_x$, equality \eqref{lem4.1} shows that $x_0\in M_g$,
and consequently, $\cap_{j=1}^nM_{f_j}\subset M_g$.
Therefore, we conclude that $M_g=\cap_{j=1}^nM_{f_j}$, as claimed.\smallskip

For each $z\in M_g=\cap_{j=1}^nM_{f_j}$,
we have $g(z)=\la=f_j(z)$, that is, $g=f_j$ on $M_g$ for each $j=1,2,\ldots,n$.
If we apply Lemma~\ref{lem2}, we deduce that $\del(g)=\del(f_j)=u_j$ on $M_{\del(g)}$.
Then $|u_j(\zeta)|=|\del(g)(\zeta)|=1$ for each $\zeta\in M_{\del(g)}$, and consequently
$\cap_{j=1}^n|u_j|^{-1}(1)\neq\emptyset$, as claimed.
\end{proof}

We explore next the intersection of the non-empty set in the previous lemma with the Choquet boundary of $B$.

\begin{lem}\label{lem5}
The intersection $\displaystyle \chb\cap\left(\bigcap_{u\in\del(\la P_x)}|u|^{-1}(1)\right)$ is non-empty
for each $\la\in\T$ and each $x\in\cha$.
\end{lem}

\begin{proof}
Let $\la\in\T$ and $x\in\cha$.
There exists $y_0\in\cap_{u\in\del(\la P_x)}|u|^{-1}(1)$
by Lemma~\ref{lem4}.
Take an arbitrary $u\in\del(\la P_x)$ (in particular, $|u(y_0)|=1=\Vinf{u}$). Define the function $\tilde{u}\in B$ by $$\tilde{u}(y)=\left(\ov{u(y_0)}^2u^2(y)+\ov{u(y_0)}u(y)\right)/2, \  (y\in Y).$$
We observe that $\tilde{u}\in Q_{y_0}$. Namely, $1=\tilde{u}(y_0)\leq\Vinf{\tilde{u}}\leq 1$, and thus $\tilde{u}\in W_{y_0}$.
Suppose now that $|\tilde{u}(y)|=1$ for some $y\in Y$,
and then $|\ov{u(y_0)}u^2(y)+u(y)|=2$.
It follows that
$$
2\leq|\ov{u(y_0)}u^2(y)|+|u(y)|\leq 2,
$$
which shows that $|u(y)|=1$.
Hence $|\ov{u(y_0)}u(y)+1|=2$, and consequently $\ov{u(y_0)}u(y)=1$.
This implies that $\tilde{u}(y)=1$, and we have  therefore proven that $\tilde{u}\in Q_{y_0}$.
We see that $(\tilde{u})^{-1}(1)$ is a peak set for $B$ with
$$(\tilde{u})^{-1}(1)=u^{-1}(u(y_0))\subset|u|^{-1}(1).$$
By the arbitrariness of $u\in\del(\la P_x)$,
we get $y_0\in\cap_{u\in\del(\la P_x)}(\tilde{u})^{-1}(1)$.
It is known that every non-empty weak peak set for $B$ contains a weak peak point,
that is, $\chb\cap\left(\cap_{u\in\del(\la P_x)}(\tilde{u})^{-1}(1)\right)\neq\emptyset$
(see, for example, \cite[Proposition~2.1]{miu1}).
This shows that $$\chb\cap\left(\cap_{u\in\del(\la P_x)}|u|^{-1}(1)\right)\neq\emptyset.$$
\end{proof}

In the next result we replace $\la P_x$ with $\la V_x$.

\begin{lem}\label{lem6}
The intersection $\displaystyle \chb\cap\left(\bigcap_{v\in\del(\la V_x)}|v|^{-1}(1)\right)$ is non-empty
for each $\la\in\T$ and each $x\in\cha$.
\end{lem}

\begin{proof} By Lemma~\ref{lem5}, there exists
$y\in\chb\cap\left(\bigcap_{u\in\del(\la P_x)}|u|^{-1}(1)\right)$.
Take an arbitrary $v\in\del(\la V_x)$.
We shall prove that $|v(y)|=1$.
Let $f\in\la V_x$ be such that $\del(f)=v$, and then $f(x)=\la$ and $\Vinf{f}=1$.
Define the function $\tilde{f}\in\SA$ by $$\tilde{f}(z)=(\ov{\la}^2f^2(z)+\ov{\la}f(z))/2, \ (z\in X).$$ We see that $\tilde{f}\in P_x$ with
$$
M_{\tilde{f}}=\set{z\in\cha:|\tilde{f}(z)|=1}=\set{z\in\cha:f(z)=\la}.
$$
Recall that $M_{\tilde{f}}=\set{z\in\cha:\tilde{f}(z)=1}$, since $\tilde{f}\in P_x$.
For each $z\in M_{\tilde{f}}$, we have $\la \tilde{f}(z)=\la=f(z)$, and thus
$\la \tilde{f}=f$ on $M_{\tilde{f}}=M_{\la \tilde{f}}$. Lemma~\ref{lem2} shows that
$$
\del(\la \tilde{f})=\del(f)
\q\mbox{on $M_{\del(\la \tilde{f})}$}.
$$
Since $\la \tilde{f}\in\la P_x$, we obtain $|\del(\la \tilde{f})(y)|=1$, that is,
$y\in M_{\del(\la \tilde{f})}$.
It follows that $v(y)=\del(f)(y)=\del(\la \tilde{f})(y)$,
and consequenlty, $|v(y)|=|\del(\la \tilde{f})(y)|=1$.
Hence $y\in|v|^{-1}(1)$. We conclude from the arbitrariness of $v\in\del(\la V_x)$ that
$y\in\chb\cap\left(\cap_{v\in\del(\la V_x)}|v|^{-1}(1)\right)$.
\end{proof}

We determine next the behaviour of $\Delta$ on sets of the form $\la P_x$.

\begin{lem}\label{lem7}
For each $(\la,x)\in\T \times \cha$, there exists a couple $(\mu,y)$ in $\T\times \chb$ such that
$\del(\la P_x)\subset\mu W_y$.
\end{lem}

\begin{proof}
Let us fix $\la\in\T$ and $x\in\cha$. By Lemma~\ref{lem6}, there exists
$y\in\chb\cap\left(\cap_{v\in\del(\la V_x)}|v|^{-1}(1)\right)$.
For each $f\in\la P_x$, we have $|\del(f)(y)|=1$
by the choice of $y$.
Since $f\in\la P_x\subset\SA$, we obtain $\Vinf{\del(f)}=1$.
Hence, $\del(f)\in\mu W_y$ with $\mu=\del(f)(y)\in\T$.\smallskip

Now, we prove that $\del(f)(y)=\del(g)(y)$ for all $f,g\in\la P_x$.
Set $h=(f+g)/2\in A$, and then $h\in\la P_x$ by Lemma~\ref{lem3}.
We observe that
$$
M_h=h^{-1}(\la)\cap\cha=f^{-1}(\la)\cap g^{-1}(\la)\cap\cha,
$$
since $f,g,h\in\la P_x$, where $k^{-1}(\la)=\set{z\in X:k(z)=\la}$
for $k\in\la P_x$.
Therefore, we have $f=h=g$ on $M_h$.
We derive  from Lemma~\ref{lem2} that $\del(f)=\del(h)=\del(g)$ on $M_{\del(h)}$.
Since $\del(h)\in\del(\la V_x)$, we get $|\del(h)(y)|=1$ by the choice of $y$.
Thus, $y\in M_{\del(h)}$, and consequently $\del(f)(y)=\del(g)(y)$.\smallskip

The above arguments show that $\del(f)\in\mu W_y$ for all $f\in\la P_x$,
where $\mu=\del(f)(y)$ is independent of the choice of $f\in\la P_x$.
This shows that $\del(\la P_x)\subset\mu W_y$
for some $\mu\in\T$ and $y\in\chb$.
\end{proof}

\begin{lem}\label{lem8}
For each $(\la,x)\in\T \times\cha$, there exists a couple  $(\mu,y)$ in $\T \times \chb$ such that
$\del(\la V_x)\subset\mu W_y$.
\end{lem}

\begin{proof} Fix $\la,x$ as in the statement.
By Lemma~\ref{lem7}, there exist $\mu\in\T$ and $y\in\chb$ such that
$\del(\la P_x)\subset\mu W_y$.
Let $v\in\del(\la V_x)$.
We shall prove that $v\in\mu W_y$.
Let $f\in\la V_x$ be such that $\del(f)=v$.
Define the function $\tilde{f}\in A$ by $$\tilde{f} (z)=(\ov{\la}^2f^2(z)+\ov{\la}f(z))/2, \ \ (z\in X).$$ We see that $\tilde{f}\in P_x$ with
$$M_{\tilde{f}}=\set{z\in\cha:\tilde{f}(z)=1}=f^{-1}(\la)\cap\cha.$$
For each $z\in M_{\tilde{f}}$, we have $\la \tilde{f}(z)=\la=f(z)$, and hence
$\la \tilde{f}=f$ on $M_{\tilde{f}}=M_{\la \tilde{f}}$.
Lemma~\ref{lem2} shows that $\del(\la \tilde{f})=\del(f)$ on $M_{\del(\la \tilde{f})}$.
Since $\tilde{f}\in P_x$, we have $\del(\la \tilde{f})\in\del(\la P_x)\subset\mu W_y$.
Thus $\del(\la \tilde{f})\in\mu W_y$, that is,
$\del(\la \tilde{f})(y)=\mu$.
This implies that $|\del(\la \tilde{f})(y)|=1$, which yields
$y\in M_{\del(\la \tilde{f})}$.
Therefore, $v(y)=\del(f)(y)=\del(\la \tilde{f})(y)=\mu$, and consequently
$v\in\mu W_y$.
This shows that $\del(\la V_x)\subset\mu W_y$.
\end{proof}

We shall discuss next the uniqueness of the couple $(y,\mu)$ in previous lemmata.

\begin{lem}\label{lem9}
If $\la V_x\subset\la' V_{x'}$ holds for some $\la,\la'\in\T$ and
$x,x'\in\cha$, then $\la=\la'$ and $x=x'$.
\end{lem}

\begin{proof}
Suppose, on the contrary, that $x\neq x'$. There exists $f\in P_x\subset V_x$
such that $|f(x')|<1$ (cf. the properties in page \pageref{properties of the Choquet boundary}). Then $\la f\in\la V_x\setminus(\la' V_{x'})$,
since $|\la f(x')|<1$.
This contradicts $\la V_x\subset\la' V_{x'}$. Hence, we obtain $x=x'$,
and thus $\la V_x\subset\la' V_x$ by the hypothesis.
For each $g\in V_x$, we have $\la g\in\la' V_x$, which shows that
$\la=\la g(x)=\la'$. We thus conclude that $\la=\la'$.
\end{proof}

\begin{lem}\label{lem10}
For each $(\la,x)\in\T \times \cha$, there exists a unique couple $(\mu,y)$ in $\T\times \chb$
such that $\del(\la V_x)=\mu W_y$.
\end{lem}

\begin{proof}  Let us fix  $\la\in\T$ and $x\in\cha$.
By Lemma~\ref{lem8} there exist $\mu\in\T$ and $y\in\chb$ such that $\del(\la V_x)\subset\mu W_y$.
Another application of Lemma~\ref{lem8}, with $\mu\in\T$, $y\in\chb$ and
$\deli$, shows the existence of
$\la'\in\T$ and $x'\in\cha$ such that $\deli(\mu W_y)\subset\la' V_{x'}$.
Thus, we have $\del(\la V_x)\subset\mu W_y\subset\del(\la' V_{x'})$,
and hence $\la V_x\subset\la' V_{x'}$.
Therefore, we obtain $\la=\la'$ and $x=x'$ by Lemma~\ref{lem9},
which shows that $\del(\la V_x)=\mu W_y$.\smallskip

Suppose that $\del(\la V_x)=\mu'W_{y'}$ for some $\mu'\in\T$ and $y'\in\chb$.
Then $\mu W_y=\del(\la V_x)=\mu'W_{y'}$, and hence $\mu W_y=\mu'W_{y'}$.
Lemma~\ref{lem9} shows that $\mu=\mu'$ and $y=y'$,
which proves the uniqueness of $\mu\in\T$ and $y\in\chb$.
\end{proof}

We are now in a position to define the key functions describing the behaviour of $\Delta$ on sets of the form $\lambda V_x$.

\begin{defn}
By Lemma~\ref{lem10}, there exist well-defined maps
$\al\colon\T\times\cha\to\T$
and $\phi\colon\T\times\cha\to\chb$ with the following property:
$$
\del(\la V_x)=\al(\la,x)W_{\phi(\la,x)}
\qq(\la\in\T,\,x\in\cha).
$$
\end{defn}

Our next goal will consist in isolating the key properties of the just defined maps.

\begin{lem}\label{lem11}
For each $\mu,\mu'\in\T$ and $y,y'\in\chb$ with $y\neq y'$,
there exist $u\in\mu Q_y$ and $v\in\mu'Q_{y'}$ such that
$\Vinf{u-v}<\sqrt{2}$.
\end{lem}

\begin{proof}
Choose disjoint open sets
$O,O'\subset Y$ so that $y\in O$ and $y'\in O'$.
There exist $u\in\mu Q_y$ and $v\in\mu'Q_{y'}$ such that $|u|<1/3$ on $Y\setminus O$
and $|v|<1/3$ on $Y\setminus O'$.
For $z\in O$, we have $|u(z)-v(z)|\leq1+1/3<\sqrt{2}$, since $O\cap O'=\emptyset$.
For $z\in Y\setminus O$, we obtain $|u(z)-v(z)|\leq1/3+1<\sqrt{2}$ by the choice of $u$.
We thus conclude $\Vinf{u-v}<\sqrt{2}$, as is claimed.
\end{proof}

\begin{lem}\label{lem12}
If $\la\in\T$ and $x\in\cha$, then $\phi(\la,x)=\phi(-\la,x)$.
\end{lem}

\begin{proof}
Let $\la\in\T$ and $x\in\cha$.
We set $\mu=\al(\la,x)$, $\mu'=\al(-\la,x)$, $y=\phi(\la,x)$ and $y'=\phi(-\la,x)$.
Then $\del(\la V_x)=\mu W_y$ and $\del((-\la)V_x)=\mu'W_{y'}$.
Suppose, on the contrary, that $y\neq y'$. Lemma~\ref{lem11} assures the existence of $\tilde{u}\in\mu Q_y$ and $\tilde{v}\in\mu'Q_{y'}$
such that $\Vinf{\tilde{u}-\tilde{v}}<\sqrt{2}$.
By the choice of $\tilde{u}$ and $\tilde{v}$, we see that
$\deli(\tilde{u})\in\deli(\mu Q_y)\subset\deli(\mu W_y)=\la V_x$
and $\deli(\tilde{v})\in\deli(\mu'W_{y'})\subset(-\la)V_x$.
Then $\deli(\tilde{u})(x)=\la$ and $\deli(\tilde{v})(x)=-\la$, and therefore
\begin{align*}
2
&=
|2\la|=|\deli(\tilde{u})(x)-\deli(\tilde{v})(x)|\leq\Vinf{\deli(\tilde{u})-\deli(\tilde{v})}\\
&=
\Vinf{\tilde{u}-\tilde{v}}<\sqrt{2},
\end{align*}
which is a contradiction. Consequently, we have $y=y'$,
and hence $\phi(\la,x)=\phi(-\la,x)$.
\end{proof}

\begin{lem}\label{lem13}
If $\la\in\T$ and $x\in\cha$, then $\phi(\la,x)=\phi(1,x)$;
hence, the point $\phi(\la,x)$ is independent of the choice of $\la\in\T$.
\end{lem}

\begin{proof}
Let $\la\in\T$ and $x\in\cha$.
Set $\mu=\al(\la,x)$, $\mu'=\al(1,x)$, $y=\phi(\la,x)$ and $y'=\phi(1,x)$.
Then $\del(\la V_x)=\mu W_y$ and $\del(V_x)=\mu'W_{y'}$.
We shall prove that $y=y'$.
Suppose that $y\neq y'$. Under this assumption, there exist $\tilde{u}\in\mu Q_y$ and $\tilde{v}\in\mu'Q_{y'}$ such that
$\Vinf{\tilde{u}-\tilde{v}}<\sqrt{2}$ (cf. Lemma~\ref{lem11}).
By the choice of $\tilde{u}$ and $\tilde{v}$, we obtain
$\deli(\tilde{u})\in\la V_x$ and $\deli(\tilde{v})\in V_x$.
Thus $\deli(\tilde{u})(x)=\la$ and $\deli(\tilde{v})(x)=1$.
If $\Re\la\leq0$, then $|\la-1|\geq\sqrt{2}$, which shows that
\begin{align*}
\sqrt{2}
&\leq
|\la-1|=|\deli(\tilde{u})(x)-\deli(\tilde{v})(x)|\leq\Vinf{\deli(\tilde{u})-\deli(\tilde{v})}\\
&=
\Vinf{\tilde{u}-\tilde{v}}<\sqrt{2}.
\end{align*}
We arrive at a contradiction, which yields $y=y'$ if $\Re\la\leq0$.
Now we consider the case when $\Re\la >0$.
Note that $\phi(-\la,x)=\phi(\la,x)=y$ by Lemma~\ref{lem12}.
Hence, $\del((-\la)V_x)=\nu W_y$ for some $\nu\in\T$.
Since $\Re(-\la)<0$, the above arguments can be applied  to
$\del((-\la)V_x)=\nu W_y$ and $\del(V_x)=\mu'W_{y'}$ to deduce that  $y=y'$. Then we get $y=y'$ even if $\Re\la >0$.
\end{proof}

\begin{defn}
By Lemma~\ref{lem13}, we may and do write $\phi(\la,x)=\phi(x)$.
We will also write $\al(\la,x)=\alx(\la)$ for each $\la\in\T$ and $x\in\cha$.
Then we obtain
\begin{equation}\label{W}
\del(\la V_x)=\alx(\la)W_{\phi(x)}
\qq(\la\in\T,\,x\in\cha).
\end{equation}
The arguments above can be applied to the surjective isometry
$\deli$ from $\SB$ onto $\SA$.
Then there exist two maps
$\be\colon\T\times\chb\to\T$ and $\psi\colon\chb\to\cha$
such that
\begin{equation}\label{V}
\deli(\mu W_y)=\be_y(\mu)V_{\psi(y)}
\qq(\mu\in\T,\,y\in\chb),
\end{equation}
where $\be_y(\mu)=\be(\mu,y)$ for each $\mu\in\T$ and $y\in\chb$.
We may regard $\alx$ and $\be_y$ as maps from $\T$ into itself
for each $x\in\cha$ and $y\in\chb$.
\end{defn}

\begin{lem}\label{lem14}
The maps $\alx\colon\T\to\T$, for each $x\in\cha$, and $\phi\colon\cha\to\chb$
are both bijective with $\alx^{-1}=\be_{\phi(x)}$ and $\phi^{-1}=\psi$.
\end{lem}

\begin{proof}
Let $x\in\cha$.
We will prove that $\alx$ and $\phi$ are injective.
Take $\la\in\T$ arbitrarily.
If we apply \eqref{V} with $\mu=\alx(\la)$ and $y=\phi(x)$, then we get
$$
\deli(\alx(\la)W_{\phi(x)})
=\be_{\phi(x)}(\alx(\la))V_{\psi(\phi(x))}.
$$
Combining the equality above with \eqref{W}, we obtain
$$
\la V_x=\deli(\alx(\la)W_{\phi(x)})=\be_{\phi(x)}(\alx(\la))V_{\psi(\phi(x))}.
$$
Lemma~\ref{lem9} shows that $\la=\be_{\phi(x)}(\alx(\la))$ and $x=\psi(\phi(x))$;
since $\la\in\T$ is arbitrary,
the first equality shows that $\alx$ is injective.
The second one shows that
$\phi$ is injective, since $x\in\cha$ is arbitrary.

Now we prove that $\alx$ and $\phi$ are both surjective.
Let $\mu\in\T$ and $y\in\chb$.
Applying \eqref{W} with $\la=\be_y(\mu)$ and $x=\psi(y)$, we get
$\del(\be_y(\mu)V_{\psi(y)})=\al_{\psi(y)}(\be_y(\mu))W_{\phi(\psi(y))}$.
The last equality, together with \eqref{V}, shows that
$$
\mu W_y=\al_{\psi(y)}(\be_y(\mu))W_{\phi(\psi(y))}.
$$
According to Lemma~\ref{lem9}, we have
\begin{equation}\label{lem13.1}
\mu=\al_{\psi(y)}(\be_y(\mu))
\end{equation}
and $y=\phi(\psi(y))$;
since $y\in\chb$ is arbitrary, the second equality shows that $\phi$ is surjective.
Then there exists $\phi^{-1}\colon\chb\to\cha$.
We obtain $\phi(\phi^{-1}(y))=y=\phi(\psi(y))$, which yields
$\phi^{-1}(y)=\psi(y)$.
We conclude, from the arbitrariness of $y\in\chb$, that $\phi^{-1}=\psi$.
Since $\psi$ is bijective with $\psi^{-1}=\phi$, for each $x\in\cha$ there exists $y\in\chb$
such that $x=\psi(y)$.
By \eqref{lem13.1}, $\mu=\al_{\psi(y)}(\be_y(\mu))=\alx(\be_{\phi(x)}(\mu))$
holds for all $\mu\in\T$.
This implies that $\alx$ is surjective for each $x\in\chb$.
There exists $\alx^{-1}$, and then
$\alx(\alx^{-1}(\mu))=\mu=\alx(\be_{\phi(x)}(\mu))$ for all $\mu\in\T$.
This shows that $\alx^{-1}=\be_{\phi(x)}$ for each $x\in\cha$.
\end{proof}

\begin{lem}\label{lem15}
For each $x\in\cha$, the map $\alx\colon\T\to\T$ is a surjective isometry.
\end{lem}

\begin{proof}
Let $x\in\cha$ and $\la_1,\la_2\in\T$.
Note that $\del(\la f)(\phi(x))=\alx(\la)$ for all $\la\in\T$ and $f\in V_x$ by \eqref{W}.
For each $f\in V_x$, we have
\begin{align*}
|\alx(\la_1)-\alx(\la_2)|
&=
|\del(\la_1f)(\phi(x))-\del(\la_2f)(\phi(x))|\\
&\leq
\Vinf{\del(\la_1f)-\del(\la_2f)}
=\Vinf{\la_1f-\la_2f}\\
&=
|\la_1-\la_2|.
\end{align*}
Hence, $|\alx(\la_1)-\alx(\la_2)|\leq|\la_1-\la_2|$. By applying the same argument to $\Delta^{-1}$ we deduce that $\be_y$ also is a contractive mapping. Having in mind that $\alx^{-1}=\be_{\phi(x)}$ (cf. Lemma~\ref{lem14}), we obtain that $\alx$ and $\be_{\phi(x)}$ are surjective isometries on $\T$.
\end{proof}

Fix $x\in\cha$. Since $\alx : \T \to \T$ is a surjective isometry on the unit sphere of the complex plane, and Tingley's problem admits a positive solution in this case, $\alx$ admits an extension to a surjective real linear isometry on $\mathbb{C}$, therefore one of the following statements hold: \begin{equation}\label{lem17}\hbox{ $\alx(\la)=\alx(1)\la$ for all $\la\in\T$, or
	$\alx(\la)=\alx(1)\ov{\la}$ for all $\la\in\T$.}
\end{equation}

One final technical result separates us from the main goal of this section.

\begin{lem}\label{lem18}
Let $f\in\SA$ and $x_0\in\cha$ be such that $|f(x_0)|<1$.
We set $\la=f(x_0)/|f(x_0)|$ if $f(x_0)\neq 0$,
and  $\la=1$ if $f(x_0)=0$.
For each $r$ with $0<r<1$, there exists $g_r\in V_{x_0}$ such that
$$
rf+(1-r|f(x_0)|)\la g_r\in\la V_{x_0}.
$$
\end{lem}

\begin{proof}
Note first that $1-|f(x_0)|>0$. We set
$$\begin{aligned}
F_0&=\left\{x\in X:|f(x)-f(x_0)|\geq\fr{1-|f(x_0)|}{2}\right\},
\q\mbox{and} \\
F_m&=\left\{x\in X:\fr{1-|f(x_0)|}{2^{m+1}}\leq|f(x)-f(x_0)|\leq\fr{1-|f(x_0)|}{2^m}\right\}
\end{aligned}$$
for each $m\in\N$.
We see that $F_n$ is a closed subset of $X$ with $x_0\not\in F_n$ for all
$n\in\N\cup\set{0}$.
Since $x_0\in\cha$, for each $n\in\N\cup\set{0}$
there exists $f_n\in P_{x_0}$ such that
\begin{equation}\label{lem18.1}
|f_n|<\fr{1-r}{1-r|f(x_0)|}
\q\mbox{on $F_n$.}
\end{equation}
We set $g_r=f_0\sum_{n=1}^\infty f_n/2^n$ (we note that the series converges in $A$).
We observe that
$$
1=g_r(x_0)\leq\Vinf{g_r}\leq\Vinf{f_0}\,\sum_{n=1}^\infty\fr{\Vinf{f_n}}{2^n}=1,
$$
and hence $g_r\in V_{x_0}$.
Set $h_r=rf+(1-r|f(x_0)|)\la g_r\in A$. We shall prove that $h_r\in\la V_{x_0}$.
Since $g_r(x_0)=1$ and $f(x_0)=|f(x_0)|\la$, we have $h_r(x_0)=\la$.
Take $x\in X$ arbitrarily.
To prove that $|h_r(x)|\leq1$, we will consider three cases.
If $x\in F_0$, then \eqref{lem18.1} shows that
$$
|g_r(x)|\leq|f_0(x)|\,\sum_{n=1}^\infty\fr{|f_n(x)|}{2^n}<\fr{1-r}{1-r|f(x_0)|}.
$$
We obtain
$$
|h_r(x)|\leq r |f(x)| + (1-r|f(x_0)|) |\la g_r(x)|<r+(1-r)=1.
$$
Hence, $|h_r(x)|<1$ if $x\in F_0$.\smallskip

Suppose that $x\in F_m$ for some $m\in\N$.
Then $|f(x)-f(x_0)|\leq(1-|f(x_0)|)/2^m$
by the choice of $F_m$.
We get
$$
|f(x)|\leq|f(x_0)|+\fr{1-|f(x_0)|}{2^m}
=\left(1-\fr{1}{2^m}\right)|f(x_0)|+\fr{1}{2^m}.
$$
We derive from \eqref{lem18.1} that
\begin{align*}
|g_r(x)|
&\leq
|f_0(x)|\left(\fr{|f_m(x)|}{2^m}+\sum_{n\neq m}\fr{|f_n(x)|}{2^n}\right)\\
&<
\fr{1}{2^m}\,\fr{1-r}{1-r|f(x_0)|}+1-\fr{1}{2^m}.
\end{align*}
It follows that
$$
|(1-r|f(x_0)|)\la g_r(x)|
<\fr{1-r}{2^m}+\left(1-\fr{1}{2^m}\right)(1-r|f(x_0)|).
$$
We infer from these inequalities that
\begin{align*}
|h_r(x)|
&\leq
r|f(x)|+|(1-r|f(x_0)|)\la g_r(x)|\\
&<
r\left(1-\fr{1}{2^m}\right)|f(x_0)|+\fr{r}{2^m}
+\fr{1-r}{2^m}+\left(1-\fr{1}{2^m}\right)(1-r|f(x_0)|)\\
&=
1,
\end{align*}
and hence, $|h_r(x)|<1$ for $x\in\cup_{n=1}^\infty F_n$.\smallskip

Now we consider the case in which  $x\not\in\cup_{n=0}^\infty F_n$.
Then $x\in\cap_{n=0}^\infty(X\setminus F_n)$, which implies that $f(x)=f(x_0)$.
We have
$$
|h_r(x)|\leq r|f(x_0)|+1-r|f(x_0)|=1,
$$ and we thus conclude that $|h(x)|\leq1$ for all $x\in X$, and consequently
$h_r\in\la V_{x_0}$.
\end{proof}

We have alrady gathered the tools to prove Theorem~\ref{thm1}.

\begin{proof}[Proof of Theorem~\ref{thm1}]
Let $f\in\SA$ and $y\in\chb$. To simplify the notation we set $x=\psi(y)$ and $\la=f(x)/|f(x)|\in\T$ if $f(x)\neq0$,
and $\la=1$ if $f(x)=0$,
where $\psi=\phi^{-1}$ as in Lemma~\ref{lem14}.
We first prove that $|\del(f)(y)|=|f(x)|$.
If $|f(x)|=1$, then $f\in\la V_x$ and thus
\[
|\del(f)(y)|=|\del(f)(\phi(x))|=|\alx(\la)|=1=|f(x)|
\]
by \eqref{W}.
We need to consider the case when $|f(x)|<1$.
By Lemma~\ref{lem18}, for each $r$ with $0<r<1$ there exists $g_r\in V_x$
such that $h_r=rf+(1-r|f(x)|)\la g_r\in\la V_x$.
We obtain
\begin{align*}
\Vinf{h_r-f}
&=
\Vinf{(r-1)f+(1-r|f(x)|)\la g_r}\\
&\leq
(1-r)+1-r|f(x)|
=2-r-r|f(x)|.
\end{align*}
Since $h_r\in\la V_x$,
we have $\del(h_r)(y)=\del(h_r)(\phi(x))=\alx(\la)$ by \eqref{W}.
Therefore, we get
\begin{align*}
1-|\del(f)(y)|
&=
|\alx(\la)|-|\del(f)(y)|
\leq
|\alx(\la)-\del(f)(y)|\\
&=
|\del(h_r)(y)-\del(f)(y)|\\
&\leq
\Vinf{\del(h_r)-\del(f)}=\Vinf{h_r-f}\\
&\leq2-r-r|f(x)|.
\end{align*}
Since $r$ with $0<r<1$ is arbitrary, we get
\begin{equation}\label{thm1.1}
1-|\del(f)(y)|
\leq|\alx(\la)-\del(f)(y)|
\leq1-|f(x)|,
\end{equation}
which shows that $|f(x)|\leq|\del(f)(y)|=|\del(f)(\phi(x))|$.
By similar arguments, applied to $\deli$ instead of $\del$, we have
$|u(y)|\leq|\deli(u)(\psi(y))|$ for all $u\in\SB$.
In particular, $|\del(f)(y)|\leq|\deli(\del(f))(\psi(y))|=|f(x)|$, and consequently
$$|\del(f) (y)|=|f(\psi (y))|, \hbox{ for all } y\in \chb, \ f\in A.$$

We shall prove that \begin{equation}\label{eq alphax module of f} \del(f)(y)=\alx(\la)|f(x)|.
\end{equation}
Since $|\del(f)(y)|=|f(x)|$, we need to consider the case when $\del(f)(y)\neq0$.
It follows from \eqref{thm1.1} that
\begin{align*}
1
&=
|\alx(\la)|
\leq|\alx(\la)-\del(f)(y)|+|\del(f)(y)|\\
&\leq
(1-|f(x)|)+|f(x)|=1,
\end{align*}
which shows that
$$
|\alx(\la)|
=|\alx(\la)-\del(f)(y)|+|\del(f)(y)|.
$$
By the equality condition for the triangle inequality,
there exists $t\geq0$ such that
$\alx(\la)-\del(f)(y)=t\del(f)(y)$.
Hence, we have $\del(f)(y)=\alx(\la)/(1+t)$.
On the other hand,
$$
|f(x)|=|\del(f)(y)|
=\left|\fr{\alx(\la)}{1+t}\right|=\fr{1}{1+t},
$$
which yields $\del(f)(y)=\alx(\la)|f(x)|$.\smallskip

Now, having in mind \eqref{lem17}, we define two subsets $\KP$ and $\KM$ of $\cha$ by
$$\begin{aligned}
\KP& =\set{x\in\cha:\alx(\lambda )= \alx(1) \lambda, \hbox{ for all } \lambda},
\q\mbox{ and } \\
\KM&=\set{x\in\cha:\alx(\lambda)= \alx(1) \overline{\lambda} \hbox{ for all } \lambda}.
\end{aligned}$$
We see that $\cha$ is the disjoint union of $\KP$ and $\KM$ (cf. \eqref{lem17}).
Recall that $\la=f(x)/|f(x)|$ if $f(x)\neq0$, and $\la=1$ if $f(x)=0$.
We derive from \eqref{eq alphax module of f} that
\begin{equation}\label{eq Delta f is a weighted composition} \begin{aligned}
		\del(f)(y)
		&=
		\alx(\la)|f(x)|=
		\begin{cases}
			\alx(1)f(x),& \hbox{ if } x\in\KP\\[2mm]
			\alx(1)\ov{f(x)},& \hbox{ if } x \in\KM
		\end{cases}\\
		&=
		\begin{cases}
			\al_{\psi(y)}(1)f(\psi(y)),&\hbox{ if } y\in\psi^{-1}(\KP)\\[2mm]
			\al_{\psi(y)}(1)\ov{f(\psi(y))},& \hbox{ if } y\in\psi^{-1}(\KM)
		\end{cases}
	\end{aligned}
\end{equation}

Set $\LP=\psi^{-1}(\KP)$ and $\LM=\psi^{-1}(\KM)$.
We deduce from the bijectivity of $\psi$ that
$\chb$ is the disjoint union of $\LP$ and $\LM$.
Consider finally, the positive homogenous extension  $T\colon A\to B$ defined by
$$
T(g)=
\begin{cases}
\D\Vinf{g}\,\del\left(\fr{g}{\Vinf{g}}\right),& \hbox{ if } g\in A\setminus\set{0}\\[2mm]
0,& \hbox{ if } g=0
\end{cases}
$$ Clearly, $T$ is a surjective mapping.
The identity in \eqref{eq Delta f is a weighted composition} shows that
$$
T(g)(y)=
\begin{cases} \al_{\psi(y)}(1) g(\psi(y)),& \hbox{ if } y\in\LP\\[2mm]
 \al_{\psi(y)}(1) \ov{g(\psi(y))},& \hbox{ if } y\in\LM
\end{cases}
\qq(g\in A).
$$
Since $\psi\colon\chb\to\cha$ is bijective, the previous identity shows that $T$ is a surjective isometry. Namely, pick $g_1,g_2$ in $A$. It follows from the previous identity and the surjectivity of $\psi$ that $$\begin{aligned}
		\|T(g_1) - T(g_2)\| &= \sup_{y\in \chb} | (T(g_1) - T(g_2)) (y)| = \sup_{y\in \chb} | (g_1 - g_2) \psi(y)| \\
		& = \sup_{x\in \cha} | (g_1 - g_2) (x)| = \|g_1-g_2\|
	\end{aligned},$$ where in the first and fourth equalities we applied that $\chb$ and $\cha$ are norming sets for $B$ and $A$, respectively (see page \pageref{properties of the Choquet boundary}).
Therefore $T$ is a real linear isometry by the Mazur--Ulam theorem.
\end{proof}

The final argument in the proof of Theorem \ref{thm1} can be also deduced from \cite[Lemma 6]{MoriOza2018} or \cite[Lemma 2.1]{FangWang06}, the identity in \eqref{eq Delta f is a weighted composition} and the fact that Choquet boundaries are boundaries, and thus norming sets.\smallskip

Although we do not make any use of the maximal convex subsets of the unit sphere of a uniformly closed function algebra, nor of the deep result asserting that a surjective isometry between the unit spheres of two Banach spaces maps maximal convex subsets to maximal convex subsets (see \cite[Lemma 5.1]{ChenDong2011} and \cite[Lemma 3.5]{Tanaka2014}),  the conclusion in \cite[Lemma 3.3]{Tan2016Mn} (see also \cite[Lemma 3.1]{HatOiTog}) can be applied to deduce that every maximal convex subset $\mathcal{C}$ of the unit sphere of uniformly closed function algebra $A$ on a locally compact Hausdorff space $X$ is of the form $$\mathcal{C} = \la V_x=\set{f\in\SA:f(x)=\la}, $$ for some $\la\in\T$ and $x\in\cha$ (this can be compared with \cite[Lemma 3.2]{HatOiTog}).

\section{Tingley's problem for commutative JB$^*$-triples}\label{sec: abelian JBstar triples}

Despite of having their own right to be studied as main protagonists, there exist certain function spaces which are also employed in other branches. An example appears in the Gelfand representation for commutative JB$^*$-triples. As a brief introduction we shall mention that these complex spaces arose in holomorphic theory, in the study and classification of bounded symmetric domains in arbitrary complex Banach spaces. These domains are the appropriate substitutes of simply connected domains to extend the Riemann mapping theorem to dimension greater than or equal to 2 (cf. \cite{Ka83} or the detailed presentation in \cite[\S 5.6]{Cabrera-Rodriguez-vol2}).\smallskip

For the sake of brevity, we shall omit a detailed presentation of the theory for general JB$^*$-triples. However, for the purpose of this note, it is worth recalling that by the Gelfand theory of JB$^*$-triples, the elements in the subclass of commutative JB$^*$-triples can be represented as spaces of continuous functions (cf. \cite{Ka83}, \cite{Zettl83}, \cite[\S 3]{BurPeRaRu2010}, \cite[\S 4.2.1]{Cabrera-Rodriguez-vol1}). Indeed, let $X$ be a subset of a Hausdorff locally convex complex space such that $0 \notin X$, $X \cup \{0\}$ is compact, and $\mathbb{T} X \subseteq X$, where $\mathbb{T} := \{\lambda \in \mathbb{T} : |\lambda| = 1\}$. Let us observe that under these hypotheses, $ \lambda x = \mu x$ for $x\in X$, $\lambda, \mu \in \mathbb{T}$ implies $\lambda = \mu$. The space $X$ is called a \emph{principal $\mathbb{T}$-bundle} in \cite{Ka83}.\smallskip

A locally compact $\mathbb{T}_{\sigma}$-space is a locally compact Hausdorff space $X$ together with a continuous mapping $\mathbb{T} \times X \to  X,$ $(\lambda, t) \mapsto \lambda t$, satisfying $\lambda (\mu t) = (\lambda \mu) t$ and $1 t = t$,  for all $\lambda, \mu \in \mathbb{T}$, $t\in X$. Each principal $\mathbb{T}$-bundle {$X$}
is a locally compact $\mathbb{T}_{\sigma}$-space. We can extend the product by elements in $\mathbb{T}$ to the one-point compactification $X\cup \{\omega\}$ of $X$ by setting $\lambda \omega = \omega$ ($\lambda \in \mathbb{T}$).
We now consider the following subspace of continuous functions on a locally compact $\mathbb{T}_{\sigma}$-space $X$
$$C^{\mathbb{T}}_0 (X) := \{ a \in C_0(X) : a (\lambda t) = \lambda a(t) \hbox{ for every } (\lambda, t) \in \mathbb{T}\times X\}.$$ We shall regard $C^{\mathbb{T}}_0 (X)$ as a norm closed subspace of $C_0 (X)$ with the supremum norm. We observe that every $C_0(L)$ space is a $C^{\mathbb{T}}_0 (X)$ space (cf. \cite[Proposition 10]{Ol74}). However, there exist examples of principal $\mathbb{T}$-bundles $X$ for which the space $C^{\mathbb{T}}_0 (X)$ is not isometrically isomorphic to a $C_0(L)$ space (cf. \cite[Corollary 1.13 and subsequent comments]{Ka83}). $C^{\mathbb{T}}_0 (X)$ spaces, with $X$ a locally compact $\mathbb{T}_{\sigma}$-space, are directly related to Lindenstrauss spaces (see \cite[Theorem 12]{Ol74}).\smallskip

Let us now fix a locally compact $\mathbb{T}_{\sigma}$-space  $X$.  Although $C^{\mathbb{T}}_0 (X)$ need not be a subalgebra of $C_0 (X)$, it is closed for the (pointwise) triple product defined by $\{a,b,c\}:= a {b}^* c= a \overline{b} c$  ($a,b,c\in C^{\mathbb{T}}_0 (X)$). We shall write $a^{[1]}=a$, $a^{[3]}= \{a,a,a\}$ and $a^{[2n +1]} = \{a,a,a^{[2n-1]}\}$ for all natural $n$. For each $t_0\in X$, the mapping $\delta_{t_0} : C^{\mathbb{T}}_0 (X)\to \mathbb{C},$ $\delta_{t_0} (a) = a(t_0)$ is a functional in the closed unit ball of the dual of $C^{\mathbb{T}}_0 (X)$. Motivated by the classical results for $C_0(L)$ spaces the reader might think that each  $\delta_{t_0}$ is an extreme point of the closed unit ball of $C^{\mathbb{T}}_0 (X)^*,$ however this is not always the case, some elements $t_0\in X$ produce zero functionals. For example if $t_0\in X$ satisfies that $t_0 \in (\mathbb{T}\backslash \{1\}) t_0$ (i.e. $t_0$ is in the $\mathbb{T}$-orbit of itself by an element which is not $1$), it is easy to check that $\delta_{t_0} =0$ as a functional in $ C^{\mathbb{T}}_0 (X)^*$. By \cite[Lemma 11]{Ol74} the extreme points of the closed unit ball of $C^{\mathbb{T}}_0 (X)^*$ are precisely those $\delta_{t_0}$ which are non-zero, that is,
\begin{equation}\label{equ extreme point in the dual ball} \partial_e\left(\mathcal{B}_{C^{\mathbb{T}}_0 (X)^*} \right) = \{ \delta_{t_0} : t_0\notin (\mathbb{T}\backslash \{1\}) t_0\}.
\end{equation}Henceforth the extreme points of the closed unit ball, $\mathcal{B}_{E}$, of a Banach space $E$ will be denoted by $\partial_e(\mathcal{B}_{E})$. Clearly, the set $\partial_e\left(\mathcal{B}_{C^{\mathbb{T}}_0 (X)^*} \right) $ is norming and a kind of Choquet boundary for $C^{\mathbb{T}}_0 (X)$.\smallskip

Those complex Banach spaces called JB$^*$-triples are precisely the complex Banach spaces whose unit ball is a bounded symmetric domain, and were introduced by W. Kaup in \cite{Ka83} to classify these domains, and to establish a generalization of Riemann mapping theorem in dimension $\geq 2$. A JB$^*$-triple is a complex Banach space $E$ admitting a continuous triple product $\{ \cdot,\cdot,\cdot\} :
E\times E\times E \to E,$ which is symmetric and linear in the outer variables, conjugate-linear in the middle one, and satisfies the following axioms:
\begin{enumerate}\item[$(a)$] $L(a,b) L(x,y) = L(x,y) L(a,b) + L(L(a,b)x,y)
	- L(x,L(b,a)y),$  for all $a,b,x,y$ in $E$,
	where $L(a,b)$ is the operator on $E$ given by $L(a,b) x = \{a,b,x\};$
	\item[$(b)$] For all $a\in E$, $L(a,a)$ is a hermitian operator with non-negative
	spectrum;
	\item[$(c)$] $\|\{a,a,a\}\| = \|a\|^3$, for all $a\in E$.\end{enumerate}
The class of JB$^*$-triples includes all C$^*$-algebras and all JB$^*$-algebras (cf. \cite[pages 522, 523 and 525]{Ka83}). \smallskip

A JB$^*$-triple $E$ is called \emph{abelian} or \emph{commutative} if the set $\{L(a,b) : a,b\in E\}$ is a commutative subset of the space $\mathcal{B}(E)$ of all bounded linear operators on $E$ (cf. \cite[\S 1]{Ka83}, \cite[\S 4.1.47]{Cabrera-Rodriguez-vol1} or \cite[\S 4]{FriRu83representation} where commutative JB$^*$-triples are called ``associative''). Despite of the technical definition, every commutative JB$^*$-triple can be isometrically represented, via a triple isomorphism (that is, a linear bijection preserving the triple product), as a space of the form $C^{\mathbb{T}}_0 (X)$ for a suitable principal $\mathbb{T}$-bundle $X$ (see \cite[Corollary 1.1]{Ka83}, \cite[Theorem 4.2.7]{Cabrera-Rodriguez-vol1}, see also the interesting representation theorems in \cite[\S 3]{FriRu82contractive} and \cite[\S 4]{FriRu83representation}).\smallskip

Let $X$ and $Y$ be two (locally compact and Hausdorff) principal $\mathbb{T}$-bundles. Each surjective linear isometry $T$ from $C^{\mathbb{T}}_0 (X)$ onto $C^{\mathbb{T}}_0 (Y)$ is a triple isomorphism (i.e., it preserves the triple product seen above). Furthermore, that is the case, if and only if, there exists a $\mathbb{T}$-equivariant homeomorphism $\phi : Y\to X$ (i.e., $\phi (\lambda s) = \lambda \phi (s)$, for all  $(\lambda, s) \in \mathbb{T}\times Y$) such that $T(a) (s) = a(\phi(s))$, for all $s\in Y$ and $a\in C^{\mathbb{T}}_0 (X)$ (see \cite[Proposition 1.12]{Ka83}). That is, surjective linear isometries and triple isomorphisms coincide, and they are precisely the composition operators with a $\mathbb{T}$-equivariant homeomorphism between the principal $\mathbb{T}$-bundles. 
\smallskip

In some of the result of this section, we can apply tools and techniques in the theory of general JB$^*$-triples. However, since the commutative objects of this category admit a concrete representation as function spaces, we strive for presenting basic arguments which do not require any knowledge on the general theory.\smallskip

Our next goal will consist in determining the explicit form of all real linear isometries between $C_0^{\mathbb{T}}(X)$ spaces for principal $\mathbb{T}$-bundles (i.e. commutative JB$^*$-triples), a description which materializes and concretizes the theoretical conclusions for real linear surjective isometries on C$^*$-algebras and JB$^*$-triples in \cite{ChuDangRussoVentura93,Dang92}.\smallskip

Let us begin by determining the triple ideals of an abelian JB$^*$-triple $C_0^{\mathbb{T}} (X)$ which are $M$-summands. This basic theory is probably known by experts but we have been unable to find a explicit reference. So, we strove for presenting a complete argument with auxiliary references. The reference book on $M$-summands and $M$-ideals is \cite{HarWerWerBookMideals}. A linear projection $P$ on a real or complex Banach space $E$ is called an \emph{$M$-projection} if $$\| x\|= \max\{\|P(x)\|, \|x - P(x)\|\}\hbox{,\quad for all  } x \in  E.$$ A closed subspace $F\subseteq E$ is called an \emph{$M$-summand} if it is the range of an $M$-projection. For a locally compact Hausdorff space $L$, the $M$-summands of $C_0(L)$ are described in  \cite[Example I.1.4$(a)$ and Lemma I.1.5]{HarWerWerBookMideals}, and they correspond to subspaces of the form $$I = \{a\in C_0(L): a(t) = 0 \hbox{ for all } t\in D\},$$ where $D$ is a clopen subset of $L$. Similar arguments to those used in the quoted reference can be applied to deduce our next lemma.

\begin{lem}\label{l Msummands in C0T spaces} Let $X$ be a principal $\mathbb{T}$-bundle. Then the $M$-summands of $C_0^{\mathbb{T}} (X)$ are precisely the subspaces of the form $$I = \{a\in C_0^{\mathbb{T}}(X): a(t) = 0 \hbox{ for all } t\in D\},$$ where $D$ is a $\mathbb{T}$-invariant  {\rm(}i.e. $\mathbb{T} D = D${\rm)} clopen subset of $X$.
	
\end{lem}	

\begin{proof} Since $X$ is a principal $\mathbb{T}$-bundle, the extreme points of the closed unit ball of $C_0^{\mathbb{T}} (X)^*$ are those in the set $\{\delta_{t_0} : t_0 \in X\}$ (cf. \eqref{equ extreme point in the dual ball}). Suppose $I$ is a closed subspace of $C_0^{\mathbb{T}} (X)$ which is an $M$-summand for an $M$-projection $P$. Let $J = (Id-P)(C_0^{\mathbb{T}} (X))$. Clearly $C_0^{\mathbb{T}}(X) = I \oplus^{\ell_\infty} J,$ and $C_0^{\mathbb{T}}(X)^* = J^{\circ} \oplus^{\ell_1} I^{\circ}$, where $I^{\circ}$ stands for the polar of $I$ in $C_0^{\mathbb{T}} (X)^*$, and similarly for $J$. Let us define $D:= \{t\in X : \delta_{t}\in I^{\circ}\}$. Clearly, $D$ is a $\mathbb{T}$-invariant closed subset of $X$ and $I\subseteq I_{D} :=\{a\in C_0^{\mathbb{T}} (X) : a (t) =0 \hbox{ for all } t\in D\}$.   The equality  $I= I_{D}$ will follow from the Hahn-Banach and Krein-Milman theorems as soon as we prove that $\partial_e(\mathcal{B}_{I^{\circ}})\subseteq I_{D}^{\circ}$, but this is clear from the well known fact that $$X\equiv \partial_e \left(\mathcal{B}_{C_0^{\mathbb{T}} (X)^*} \right) = \partial_e \left(\mathcal{B}_{I^{\circ}} \right) \cup \partial_e \left(\mathcal{B}_{J^{\circ}} \right).$$ 	
	
Similarly, $J =	\{a\in C_0^{\mathbb{T}} (X) : a (t) =0 \hbox{ for all } t\in \widehat{D}\}$ for a $\mathbb{T}$-invariant closed subset $\widehat{D}$ of $X$. Since $D$ and $\widehat{D}$ are disjoint with $D \cup \widehat{D} = X,$ we deduce that $D$ is clopen.
\end{proof}

It should be noticed that the set $D$ in the previous lemma might be compact.\smallskip

Let us now determine the surjective real linear isometries between abelian JB$^*$-triples.

\begin{lem}\label{l form of surjective real linear isometries between abelian JBstar triples} Let $X$ and $Y$ be two principal $\mathbb{T}$-bundles. Then for each surjective real linear isometry
	$T : C^{\mathbb{T}}_0 (X) \to C^{\mathbb{T}}_0 (Y)$ there exist a $\mathbb{T}$-invariant clopen subset $D\subseteq X$ and a $\mathbb{T}$-equivariant homeomorphism $\phi : Y\to X$ satisfying $$T(a) (s) = a(\phi(s)), \text{ for all } a\in C^{\mathbb{T}}_0 (X) \hbox{ and for all } s\in \phi^{-1}(D),$$ and $$T(a) (s) = \overline{a(\phi(s))}, \text{ for all } a\in C^{\mathbb{T}}_0 (X) \hbox{ and for all } s\in \phi^{-1}(X\backslash D).$$   Moreover, the set $\tilde{D}:= \phi^{-1}(D)$ is $\mathbb{T}$-invariant and clopen in $Y$, $T_1:= T|_{C^{\mathbb{T}}_0 (D)} : C^{\mathbb{T}}_0 (D) \to C^{\mathbb{T}}_0 (\tilde{D})$ is a complex linear surjective isometry given by the composition operator of the mapping $\phi|_{\tilde{D}}$, $T_2:= T|_{C^{\mathbb{T}}_0 (X\backslash D)} : C^{\mathbb{T}}_0 (X\backslash D) \to C^{\mathbb{T}}_0 (Y\backslash \tilde{D})$ is a conjugate-linear surjective isometry given by the pointwise conjugation and the composition operator of the mapping $\phi|_{Y\backslash \tilde{D}}$, $C^{\mathbb{T}}_0 (X) = C^{\mathbb{T}}_0 (D) \oplus^{\ell_{\infty}} C^{\mathbb{T}}_0 (X\backslash D),$ and $T = T_1 \oplus T_2$.
\end{lem}

\begin{proof} Since the unique Cartan factor of rank 1 which can be an $M$-summand in the second dual of an abelian JB$^*$-triple is the trivial one ($\mathbb{C}$), it follows from \cite[Theorem 3.1]{Dang92} that there exist two closed subtriples $I$ and $J$ of $C^{\mathbb{T}}_0 (X)$ such that $C^{\mathbb{T}}_0 (X) = I \oplus^{\ell_{\infty}} J$, $T_1:=T|_{I}: I\to T(I)$ is a surjective complex linear isometry and  $T_2:=T|_{J}: J\to T(J)$ is a surjective conjugate-linear isometry. \smallskip
	
By Lemma \ref{l Msummands in C0T spaces} there exist $\mathbb{T}$-invariant clopen subsets $D\subseteq X$ and $\tilde{D}\subseteq Y$ such that $I = \{ a\in  C^{\mathbb{T}}_0 (X) : a(t) =0 \hbox{ for all } t\in X\backslash D\}\cong C^{\mathbb{T}}_0 (D) $ and $T(I) = \{ a\in  C^{\mathbb{T}}_0 (Y) : a(t) =0 \hbox{ for all } t\in Y\backslash \tilde{D}\}\cong C^{\mathbb{T}}_0 (\tilde{D})$. Since $T|_{I}: I\to T(I)$ is a surjective  complex linear isometry, we derive from Kaup's theorem  (cf. \cite[Proposition 1.12]{Ka83}) that there exists a $\mathbb{T}$-equivariant homeomorphism $\phi_1: \tilde{D} \to D$  satisfying $T(a) (s) = a(\phi_1(s))$ for all $a\in I$, $s\in \tilde{D}$.\smallskip

Since, clearly, $J =\{ a\in  C^{\mathbb{T}}_0 (X) : a(t) =0 \hbox{ for all } t\in D\}\cong C^{\mathbb{T}}_0 (X\backslash D) $ and $T(J) = \{ a\in  C^{\mathbb{T}}_0 (Y) : a(t) =0 \hbox{ for all } t\in  \tilde{D}\}\cong C^{\mathbb{T}}_0 (Y\backslash \tilde{D})$, and the mapping $\overline{\ \cdot \ } \circ T_2 = T|_{J}: J\to T(J),$ $a\mapsto \overline{T(a)}$ ($a\in J$) is a surjective complex linear isometry, Kaup's theorem gives a $\mathbb{T}$-equivariant homeomorphism $\phi_2: Y\backslash \tilde{D} \to X\backslash D$  satisfying $T(a) (s) = \overline{a(\phi_{2}(s))}$ for all $a\in J$, $s\in Y\backslash \tilde{D}$.\smallskip

Finally taking $\phi : Y\to X$, $\phi(s) = \phi_1(s)$ for $s\in \tilde{D}$ and $\phi(s) = \phi_2(s)$ for $s\in Y\backslash  \tilde{D}$ we get the desired $\mathbb{T}$-equivariant homeomorphism from $Y$ onto $X$.
\end{proof}

Let $X$ be a locally compact $\mathbb{T}_{\sigma}$-space. Let us consider the commutative C$^*$-algebra $$C^{hom}_0(X) = \{b\in C_0(X) : b(\lambda x) = b(x)\ \forall x\in X, \lambda \in \mathbb{T}\}\equiv C_0(X/\mathbb{T}),$$ equipped with the pointwise product and involution and the supremum norm. The commutative JB$^*$-triple $C^{\mathbb{T}}_0 (X)$ is a Banach $C^{hom}_0(X)$-bimodule under the pointwise product.\smallskip

Let us continue with a rudimentary \emph{continuous triple functional calculus} in our setting. Let $\mathcal{B}_{\mathbb{C}}$ denote the closed unit ball of $\mathbb{C},$ regarded as principal $\mathbb{T}$-bundle. For each $a\in C^{\mathbb{T}}_0 (X)$ with $\|a\|\leq 1$ and each $f$ in $C^{\mathbb{T}}_0 (\mathcal{B}_{\mathbb{C}})= \{ f\in C( \mathcal{B}_{\mathbb{C}}) : f(0) =0, \ f(\lambda \zeta ) = \lambda f(\zeta), \ \lambda\in \mathbb{T}, \zeta \in \mathcal{B}_{\mathbb{C}} \}$, the composition $f\circ a$ lies in $ C^{\mathbb{T}}_0 (X)$, and it will be denoted by $f_t(a) = f\circ a$. This coincides with the so-called continuous triple functional calculus in the wider setting of JB$^*$-triples. Let us observe that for $f_n (\zeta) = |\zeta|^{2n} \zeta$ with $n\in \mathbb{N}\cup \{0\}$ ($\zeta\in \mathcal{B}_{\mathbb{C}}, f_n\in C^{\mathbb{T}}_0 (\mathcal{B}_{\mathbb{C}})$), and each $a$ in the closed unit ball of $C^{\mathbb{T}}_0 (X),$ we have $(f_n)_t (a) = a^{[2 n +1]}$.\smallskip

The next result is a type of concretized version of \cite[Lemma 3.3]{EdFerHosPe2010}.

\begin{lem}\label{l functional calculus on faces}
Let $F$ be a norm closed face of the closed unit ball of $C^{\mathbb{T}}_0 (X)$, where $X$ is a principal $\mathbb{T}$-bundle, and let $f$ be a function in $C^{\mathbb{T}}_0 (\mathcal{B}_{\mathbb{C}})$ such that $f$ is the identity on $\mathbb{T}$.
Then for each $a$ in $F$, the element $f_t (a)$ lies in $F$.	
\end{lem}

\begin{proof} Since each $f\in C^{\mathbb{T}}_0 (\mathcal{B}_{\mathbb{C}})$ satisfies $f(\lambda \zeta ) = \lambda f(\zeta)$ ($ \lambda\in \mathbb{T}, \zeta \in \mathcal{B}_{\mathbb{C}}$), the values of $f$ on the interval $[0,1]$ determine the whole function $f$. We can now repeat, almost literally the argument in \cite[Lemma 3.3]{EdFerHosPe2010}. Fix $a\in F$ and a positive $\varepsilon< 1/2$, let $f_{\varepsilon}$ and $g_{\varepsilon}$ denote the functions in $C^{\mathbb{T}}_0 (\mathcal{B}_{\mathbb{C}})$ whose restrictions to $[0,1]$ are given by  $$
	f_{\varepsilon}(t)= \left\{
	\begin{array}{ll}
		0, &     0 \leq t \leq \varepsilon/2,\\
		\hbox{affine}, & \varepsilon/2 \leq t \leq \varepsilon,\\
		f(t), &    \varepsilon \leq t < 1-\varepsilon, \\
		\hbox{affine}, & 1-\varepsilon \leq t \leq 1-\varepsilon/2,\\
		1, &     1- \varepsilon/2 \leq t \leq 1    \end{array}%
	\right.$$ and $$g_{\varepsilon} (t) = (1 - \varepsilon/2)^{-1} (t -
		\varepsilon/2 f_{\varepsilon}(t) ), \ t\in [0,1],$$ respectively. According to this definition, $\|f - f_{\varepsilon}\|_{\infty}$ tends to zero when $\varepsilon$ tends to zero, and $ (1 - \varepsilon/2) g_{\varepsilon}(t) + (\varepsilon/2) f_{\varepsilon}(t)= t,$ for all $t\in[0,1]$, witnessing that $a = (1 - \varepsilon/2) (g_{\varepsilon})_t(a) + (\varepsilon/2)
		(f_{\varepsilon})_t(a).$ Applying now that $F$ is a face, we conclude that $(f_{\varepsilon})_t(a)\in F$ for all $0< \varepsilon< 1/2$. Letting $\varepsilon$ tend to zero we get $f_t (a)\in F$.
\end{proof}

We fix now an arbitrary locally compact $\mathbb{T}_{\sigma}$-space $X$. Our space $C^{\mathbb{T}}_0 (X)$ lacks of peaking functions, since for each $a\in S(C^{\mathbb{T}}_0 (X))$, we have $\mathbb{T}\subseteq a(X)$.  We can combine the description of the set of extreme points of the closed unit ball of $C^{\mathbb{T}}_0 (X)^*$ given in \eqref{equ extreme point in the dual ball} with the facial theory of JB$^*$-triples in \cite{EdFerHosPe2010} to determine the maximal proper faces of the closed unit ball of $C^{\mathbb{T}}_0 (X)$, however we prioritize a self-contained argument for function spaces more accessible for a wider audience. \smallskip

We recall a fundamental property of $C^{\mathbb{T}}_0 (X)$. Let $\pmb{\mu}$ denote the unit Haar measure on $\mathbb{T}$. For each $a\in C_0(X)$ we consider a function $\pi_{\mathbb{T}} ({a}) : X\to \mathbb{C}$ defined by $$ \pi_{\mathbb{T}} ({a}) (t) = \int \lambda^{-1} {a}(\lambda t) d\pmb{\mu}, \ \ (t\in X).$$ It is known that $\pi_{\mathbb{T}}$ is a contractive projection of $C_0(X)$ onto $C_0^{\mathbb{T}} (X)$ (cf. \cite{Ol74}).

Pick $t_0\in X$ with $t_0\notin (\mathbb{T}\backslash \{1\}) t_0$ and $\mu \in \mathbb{T}$. We shall define the set $$F_{t_0,\mu} = F_{t_0,\mu}^X:=\{a\in \mathcal{B}_{C^{\mathbb{T}}_0 (X)} : a (t_0) =\mu \}.$$ At this stage it is not at all clear that $F_{t_0,\mu}$ is non-empty, an statement which is straightforward in the case of $C_0(L)$ spaces thanks to Urysohn's lemma.

\begin{remark}\label{r Urysohn's type lemma} Suppose $X$ is a locally compact $\mathbb{T}_{\sigma}$-space.
Let $W$ be a $\mathbb{T}$-invariant open neighbourhood of $t_0$ in $X$ which is contained in a compact $\mathbb{T}$-invariant subset. We consider the following continuous function $$\mathbb{T} t_0 \cup \left(X\backslash W \right)\to \mathbb{C},$$ $$\lambda t_0\mapsto \lambda, \hbox{ and } t\mapsto 0 \hbox{ for all } t\in X\backslash W.$$  Find, via Tiezte's theorem, a continuous function $\tilde{h}\in C_0(X)$ extending the previous mapping. Let $h:= \pi_{\mathbb{T}} (\tilde{h})\in C_0^{\mathbb{T}} (X)$. It is easy to check that $h(t_0) = 1$ and $h(t) = 0 $ for all $t\in X\backslash W$. Clearly, $\mu h \in F_{t_0,\mu}$. This construction, which was already considered in \cite[Proof of Lemma 11]{Ol74}, is a kind of Urysohn's lemma for $C_0^{\mathbb{T}} (X)$ spaces, and will be employed along this section.	
\end{remark}

Let $E$ be a Banach space. As observed by R. Tanaka in \cite[Lemma 3.3]{Tan2016Mn} and \cite[Lemma 3.2]{Tan2016finitevN}, Eidelheit's separation theorem or the geometric Hahn-Banach theorem can be employed to deduce that a convex subset $C\subseteq S(E)$ is a maximal convex subset if and only if it is a maximal norm closed proper face of the closed unit ball, $\mathcal{B}_{E}$, of $E$.\smallskip

We give next a concrete description of the norm closed faces of $\mathcal{B}_{C_0^{\mathbb{T}} (X)}$. The conclusion can be also derived from the study of norm closed faces of the closed unit ball of a general JB$^*$-triple \cite{EdFerHosPe2010} and a good knowledge on the minimal tripotents in the second dual and its relation with the extreme point of the closed unit ball of the first dual. For the sake of simplicity we include here an alternative argument with techniques of function algebras.

\begin{lem}\label{l max convex susbsets of C0T}
Let $X$ be a principal $\mathbb{T}$-bundle. Then every maximal convex subset of $S(C_0^{\mathbb{T}} (X))$ {\rm(}equivalently, each maximal proper norm closed face of the closed unit ball of $C_0^{\mathbb{T}} (X)${\rm)} is of the form $$F_{t_0,\mu}=F_{t_0,\mu}^X :=\{a\in \mathcal{B}_{C^{\mathbb{T}}_0 (X)} : a (t_0) =\mu \},$$ for some $t_0\in X$ and $\mu \in \mathbb{T}$.
\end{lem}

\begin{proof} Under these hypotheses $F_{t_0,\mu}$ is a non-empty proper face of the closed unit ball of $C_0^{\mathbb{T}} (X)$. Let $F$ be a maximal convex (proper) subset of $S(C_0^{\mathbb{T}} (X))$ containing $F_{t_0,\mu}$. Fix $b\in F_{t_0,\mu}$. If there exists $a\in F\backslash F_{t_0, \mu}$ the function $c = \frac{a+b}{2}\in F$ satisfies $|c(t_0)|<\delta<1$ for an appropriate $\delta$. The set $U:= \{t\in X : |c(t)|<\delta\}$ contains $t_0$, is open and $\mathbb{T}$-invariant because $c\in C_0^{\mathbb{T}} (X)$. As we commented above, we can find $h\in C_0^{\mathbb{T}}(X)$ with $\|h\|=1$, $h(t_0)= \mu$ and $h|_{X\backslash U} \equiv 0$. \smallskip

Let $k$ be the function in $C_0^{\mathbb{T}}(\mathcal{B}_{\mathbb{C}})$ whose restriction to $[0,1]$ vanishes on $[0,\delta],$ takes value $1$ at $1$ and is affine on the rest. Lemma \ref{l functional calculus on faces} assures that $k_t(c) = k\circ c\in F.$ Since $k_t(c)|_{U} \equiv 0$ it can be easily seen that $k_t(c) h = 0$. Therefore $\left\|\frac{h+k_t(c)}{2}\right\|= \max\left\{\left\|\frac{h}{2}\right\|, \left\|\frac{k_t(c)}{2}\right\|\right\} = \frac12$. Since, clearly, $\frac{h+k_t(c)}{2}\in F$, the intersection of $F$ with the open unit ball of $C_0^{\mathbb{T}} (X)$ is non-empty, which contradicts that $F$ is proper.
\end{proof}

Labelling the proper maximal norm closed faces, $F_{t_0,\mu},$ of the closed unit ball of $C_0^{\mathbb{T}} (X)$ in terms of pairs $(t_0, \mu)$ with $t_0\in X$ and $\mu \in \mathbb{T}$ does not produce an unambiguous association because $F_{t_0,\mu} =F_{\gamma t_0,\gamma \mu}$ for all $\gamma \in \mathbb{T}$. To avoid repetitions let us consider the following property: a non-empty subset $\mathcal{S}$ of a principal $\mathbb{T}$-bundle $X$ satisfies the \emph{non-overlapping property} if for each $t\in \mathcal{S}$ we have $\mathcal{S}\cap \mathbb{T} t = \{t\}.$  Thanks to Zorn's lemma we can always find a maximal non-overlapping subset $X_0$ of $X$. Let us observe that in this case, $\mathbb{T} X_0 = X$, actually, for each $t\in X$ there exist unique $t_0\in X_0$ and $\mu \in \mathbb{T}$ such that $t = \mu t_0$. Consequently, the set $\{\delta_{t_0} : t_0 \in X_0\}$ is norming.  Furthermore, the set $$\{F_{t_0, \mu }: \mu \in \mathbb{T}, t_0\in X_0\}$$ covers all possible proper maximal norm closed faces of $\mathcal{B}_{C_0^{\mathbb{T}} (X)}.$ Namely, given a maximal proper face of the form $F_{s_0, \mu_0}$ there exist unique $t_0\in X_0$ and $\nu\in \mathbb{T}$ such that $s_0 = \nu t_0$, and thus $$F_{s_0, \mu_0} = F_{\nu t_0, \mu_0} = F_{t_0, \overline{\nu} \mu_0}.$$ Actually, for each proper maximal face $F$ of $\mathcal{B}_{C_0^{\mathbb{T}}(X)}$
	\begin{equation}\label{eq unique labels for X0}  \hbox{there exist unique } t_0\in X_0 \hbox{ and } \mu \in \mathbb{T} \hbox{ such that }  F = F^X_{t_0, \mu}.
	\end{equation}

An alternative proof for Lemma \ref{l max convex susbsets of C0T} can be deduced from \cite[Lemma 3.3]{Tan2016Mn} (see also \cite[Lemma 3.1]{HatOiTog}.\smallskip

The main result of this section is a solution to Tingley's problem in the case of commutative JB$^*$-triples.

\begin{thm}\label{t Tingleys problem for commutative JBstar triples} Let $X$ and $Y$ be two principal $\mathbb{T}$-bundles. Then each surjective isometry
$\Delta : S(C^{\mathbb{T}}_0 (X)) \to S(C^{\mathbb{T}}_0 (Y))$ admits an extension to a surjective real linear isometry $T : C^{\mathbb{T}}_0 (X)\to C^{\mathbb{T}}_0 (Y)$.\smallskip 	

Furthermore, there exist a $\mathbb{T}$-invariant clopen subset $D\subseteq X$ and a $\mathbb{T}$-equivariant homeomorphism $\phi : Y\to X$ satisfying $$\Delta (a) (s) = a(\phi(s)), \hbox{ for all } a\in S(C^{\mathbb{T}}_0 (X)) \hbox{ and for all } s\in \phi^{-1}(D),$$ and $$\Delta (a) (s) = \overline{a(\phi(s))}, \hbox{ for all } a\in S(C^{\mathbb{T}}_0 (X)) \hbox{ and for all } s\in \phi^{-1}(X\backslash D).$$ Consequently, there exists a surjective isometry $T: C^{\mathbb{T}}_0 (X) \to C^{\mathbb{T}}_0 (Y)$ such that $T|_{C^{\mathbb{T}}_0 (D)}$ is complex linear, $T|_{C^{\mathbb{T}}_0 (X\backslash D)}$ is conjugate-linear and $T(a) =\Delta(a)$ for all $a\in S(C^{\mathbb{T}}_0 (X))$.
\end{thm}

The proof will be given after a series of technical lemmata. Let us begin by recalling a key result in the techniques developed to study the problem of extension of isometries which is essentially due to L. Cheng and Y. Dong \cite[Lemma 5.1]{ChenDong2011} and R. Tanaka \cite{Tan2016Mn} (see also \cite[Lemma 3.5]{Tanaka2014}, \cite[Lemmas 2.1 and 2.2]{Tan2017}).

\begin{prop}\label{p faces ChengDong11}{\rm(}\cite[Lemma 5.1]{ChenDong2011}, \cite[Lemma 3.3]{Tan2016Mn}, \cite[Lemma 3.5]{Tanaka2014}{\rm)} Let $\Delta: S(E) \to S(F)$ be a surjective isometry between the unit spheres of two Banach spaces, and let $\mathcal{M}$ be a convex subset of $S(E)$. Then $\mathcal{M}$ is a maximal proper face of $\mathcal{B}_E$ {\rm(}equivalently, a maximal convex subset of $S(E)${\rm)} if and only if $\Delta(\mathcal{M})$ is a maximal proper {\rm(}closed{\rm)} face of $\mathcal{B}_F$ {\rm(}equivalently, a maximal convex subset of $S(F)${\rm)}.
\end{prop}

The next corollary is a consequence of Proposition \ref{p faces ChengDong11} and Lemma \ref{l max convex susbsets of C0T}.

\begin{cor}\label{c first application of faces}
Let $X$ and $Y$ be two principal $\mathbb{T}$-bundles, and let $\Delta : S(C^{\mathbb{T}}_0 (X)) \to S(C^{\mathbb{T}}_0 (Y))$ be a surjective isometry. Then for each $t_0\in X$ and each $\mu \in \mathbb{T}$ there exist elements $s_0\in Y$ and $\nu\in \mathbb{T}$ satisfying $$\Delta (F_{t_0, \mu}^{X}) = F_{s_0, \nu}^{Y}.$$
\end{cor}

We have already given some arguments showing that the elements $s_0$ and $\nu$ in the conclusion of the previous corollary need not be unique. To avoid the problem we consider the next lemma.

\begin{lem}\label{l faces maximal non-overlapping}
	Let $X$ and $Y$ be two principal $\mathbb{T}$-bundles, and let $\Delta : S(C^{\mathbb{T}}_0 (X)) \to S(C^{\mathbb{T}}_0 (Y))$ be a surjective isometry.  Let $X_0$ be a maximal non-overlapping subset of $X$. Then for each $t_0$ in $X_0$ there exists a unique $s_0 = \phi (t_0)\in Y$ satisfying $$\Delta (F_{t_0, 1}^{X}) = F_{\phi(t_0), 1}^{Y}.$$  The mapping $\phi= \phi_{X_0}: X_0\to Y$ is well-defined and injective.
\end{lem}

\begin{proof} By Corollary \ref{c first application of faces} there exist $s\in Y$, $\mu\in \mathbb{T}$ such that \begin{equation}\label{eq existence of phit0} \Delta (F_{t_0,1}^X) = F_{s,\mu}^Y = F_{\overline{\mu} s, 1}^Y.
	\end{equation} Observe that the element $\overline{\mu} s$ satisfying the identity in \eqref{eq existence of phit0} is unique. Indeed, if $F_{ s_1, 1}^Y = F_{s_2, 1}^Y$ for some $s_1, s_2$ in $Y$ with $s_2\notin \mathbb{T} s_1$ we can find a function $a\in S(C_0^{\mathbb{T}}(Y))$ with $a(s_1)=1$ and $a(s_2) = 0$ (cf. Remark \ref{r Urysohn's type lemma}), which is impossible. If $s_2 = \nu s_1$ for some $\nu \in \mathbb{T}$, for each $a\in F_{ s_1, 1}^Y = F_{s_2, 1}^Y$, we have $1= a(s_1) = a(s_2) = \nu a(s_1) = \nu$, witnessing that $s_1 = s_2$. The first conclusion follows by setting $\phi (t_0) = \overline{\mu} s$ for the element  $\overline{\mu} s$ given by \eqref{eq existence of phit0}.\smallskip

The rest is clear from the previous arguments.
\end{proof}

Henceforth we fix a surjective isometry $\Delta : S(C^{\mathbb{T}}_0 (X)) \to S(C^{\mathbb{T}}_0 (Y))$, where $X$ and $Y$ are two principal $\mathbb{T}$-bundles, $X_0\subseteq X$ a maximal non-overlapping subset and $\phi= \phi_{X_0}: X_0\to Y$ the injective mapping given by Lemma \ref{l faces maximal non-overlapping}.\smallskip

The next step in our strategy isolates a crucial property of $\phi$.

\begin{lem}\label{l functions vanishing at phit0} Let $t_0$ be an element in $X_0$, and let $a$ be an element in $S(C_0^{\mathbb{T}}(X))$ satisfying $a(t_0)=0$. Then $\Delta (a) (\phi(t_0)) =0$.
\end{lem}

\begin{proof} Since a continuous function in $S(C_0^{\mathbb{T}}(X))$ vanishing at $t_0$ can be approximated in norm by functions in  $S(C_0^{\mathbb{T}}(X))$ vanishing on a neighbourhood of $t_0$, we can assume that $a$ vanishes on a $\mathbb{T}$-invariant open neighbourhood $U$ of $t_0$.\smallskip
	
Let us consider a function $b\in S(C_0^{\mathbb{T}}(X))$ satisfying $b(t_0)=1$ and $b|_{X\backslash U} \equiv 0.$	Clearly, $b\in F_{t_0,1}$, and by orthogonality of $a$ and $b$ we have $\|a\pm b\| = 1$.\smallskip

By \cite[Proposition 2.3$(a)$]{Mori2017} we have $\Delta (-F_{t_0,1}) = -\Delta (F_{t_0,1})$. Therefore, there exists $c\in F_{t_0,1}$ such that $-\Delta (c) = \Delta (-b)$, and hence $\Delta (-b) (\phi (t_0)) = -\Delta (c) (\phi(t_0)) = -1$. It follows from the assumptions that
$$\begin{aligned} 1 &= \|a+ b\| = \|\Delta (a) - \Delta(-b)\|\geq |\Delta (a) (\phi(t_0)) - \Delta(-b) (\phi(t_0)) | \\ &= |\Delta (a) (\phi(t_0)) +1 |,\\
1 &= \|a- b\| = \|\Delta (a) - \Delta(b)\|\geq |\Delta (a) (\phi(t_0)) - \Delta(b) (\phi(t_0)) | \\ &= |\Delta (a) (\phi(t_0)) -1 |
\end{aligned}$$ witnessing that $\Delta (a) (\phi(t_0)) =0$.
\end{proof}

\begin{lem}\label{Y0 is a maximal non-overlapping}
The set $Y_0 = \{\phi(t_0): t_0\in X_0\}=\phi(X_0)\subseteq Y$ is a maximal non-overlapping subset of $Y$, and hence the set $\{\delta_{s} : s\in Y_0\}$ is norming in $C_0^{\mathbb{T}} (Y)$. Furthermore, the mapping $\phi : X_0\to Y_0$ is a bijection satisfying $$\Delta(F_{t_0,1}^X) = F_{\phi(t_0),1}^Y, \hbox{ for all } t_0 \in X_0,$$ and $\phi^{-1}$ is precisely the mapping given by Lemma \ref{l faces maximal non-overlapping} for  $\Delta^{-1}$ and $Y_0$.
\end{lem}

\begin{proof} We shall first show that $Y_0$ is non-overlapping. Since $\phi$ is injective we can suppose that we have $\phi(t_1)\neq \phi(t_2)$ in $Y_0$ with $t_1\neq t_2$ in $X_0$. Since $X_0$ is non-overlapping we can find $a\in F^X_{t_1,1}$ with $a(t_2)=0$ (cf. Remark \ref{r Urysohn's type lemma}). Lemma \ref{l functions vanishing at phit0} implies that $\Delta (a) (\phi(t_2)) =0$, and consequently $\phi(t_2)\notin \mathbb{T} \phi(t_1)$ because $\Delta(a)\in \Delta (F^Y_{\phi(t_1),1})$. This shows that $Y_0$ is non-overlapping.\smallskip
	
Let us find, via Zorn's lemma, a maximal non-overlapping subset of $Y,$ $\tilde{Y}_0 \supset Y_0$. By applying Lemma \ref{l faces maximal non-overlapping} to $\Delta^{-1}$ and $\tilde{Y}_0$ we deduce the existence of an injective mapping $\psi: \tilde{Y}_0 \to X$ satisfying $$\Delta^{-1} (F^Y_{s_0,1}) = F^X_{\psi(s_0),1}, \hbox{ for all } s_0\in \tilde{Y}_0.$$ In particular we have $$\Delta(F^X_{\psi(\phi(t_0)),1}) = F^Y_{\phi(t_0),1} = \Delta (F^X_{t_0,1}), \hbox{ for all }  t_0\in X_0,$$ which implies that  $\psi(\phi(t_0)) = t_0$ for all $t_0\in X_0$.\smallskip

By the first part of our argument, applied to $\Delta^{-1}$ and $\tilde{Y}_0,$ we know that $\psi (\tilde{Y}_0)$ must be a non-overlapping subset of $X$ containing $X_0$. The maximality of $X_0$ implies that $\psi (\tilde{Y}_0) = X_0.$\smallskip

If there exists $s_3\in \tilde{Y}_0\backslash Y_0$, by the maximality of $X_0$, and the definition of $\psi$, there exist $t_0\in X_0$ and $\nu \in \mathbb{T}$ such that $$\Delta^{-1}  ( F_{s_3,1}^Y) =  F_{\psi(s_3),1}^X =  F_{\nu t_0,1}^X.$$ Having in mind that $\tilde{Y}_0$ is non-overlapping, $s_3\in \tilde{Y}_0\backslash Y_0$ and $\phi(t_0)\in Y_0$, we can find $\tilde{a}\in F_{s_3,1}^Y$ vanishing at $\phi(t_0)$ (cf. Remark \ref{r Urysohn's type lemma}). Lemma \ref{l functions vanishing at phit0}, applied to $\Delta^{-1}$, $\tilde{a}$ and $\phi(t_0)$, implies that $$0=\Delta^{-1} (\tilde{a}) (\psi\phi(t_0)) = \Delta^{-1} (\tilde{a}) (t_0).$$ However $\Delta^{-1} (\tilde{a}) \in  \Delta^{-1}  ( F_{s_3,1}^Y) =  F_{\nu t_0,1}^X$, and hence $1= \Delta^{-1} (\tilde{a}) (\nu t_0 ) = \nu \Delta^{-1} (\tilde{a}) ( t_0 ) =0$, leading to a contradiction.
\end{proof}

Having in mind that $X_0$ and $Y_0$ are maximal non-overlapping subsets of $X$ and $Y$, respectively, and considering the property we commented in \eqref{eq unique labels for X0}, we deduce the next fact:
\begin{equation}\label{eq definition tildephi sigma}\begin{aligned} &\hbox{for each } t_0\in X_0, \hbox{ {and} } \mu\in \mathbb{T} \hbox{ there exist unique } \tilde{\phi} (t_0,\mu)\in Y_0 \\ &\hbox{ and } \sigma(t_0, \mu)\in \mathbb{T} \hbox{ satisfying } \Delta(F^{X}_{t_0, \mu}) = F^{Y}_{\tilde{\phi} (t_0,\mu), \sigma(t_0, \mu)}.
\end{aligned}
\end{equation}
The mappings $\tilde{\phi} : X_0\times \mathbb{T}\to Y_0$ and $\sigma : X_0\times \mathbb{T}\to \mathbb{T}$ are well-defined. We shall make use of these mappings in the subsequent results.

\begin{lem}\label{l sigma is odd and tildephi is even} The mappings $\sigma$ and $\tilde{\phi}$ satisfy $$\sigma (t_0, -\mu) = -\sigma (t_0, \mu), \hbox{ and } \tilde{\phi} (t_0, -\mu) = \tilde{\phi} (t_0, \mu),$$
for all $t_0\in X_0$ and  $\mu\in \mathbb{T}$. 	
\end{lem}	

\begin{proof} A new application of \cite[Proposition 2.3$(a)$]{Mori2017} gives $$\begin{aligned}
F^{Y}_{\tilde{\phi} (t_0,-\mu), \sigma(t_0, -\mu)} &= \Delta(F^{X}_{t_0, -\mu}) = \Delta(- F^{X}_{t_0, \mu}) = - \Delta(F^{X}_{t_0, \mu}) \\
&= - F^{Y}_{\tilde{\phi} (t_0,\mu), \sigma(t_0, \mu)} = F^{Y}_{\tilde{\phi} (t_0,\mu), -\sigma(t_0, \mu)}.
	\end{aligned}$$
	
If 	$\tilde{\phi} (t_0,-\mu)\neq \tilde{\phi} (t_0,\mu)$ in $Y_0$, there exists a function $\tilde{a}$ in $F^{Y}_{\tilde{\phi} (t_0,-\mu), \sigma(t_0, -\mu)}$ vanishing at $\tilde{\phi} (t_0,\mu)$ (cf. Remark \ref{r Urysohn's type lemma}), contradicting the previous identity. Therefore $\tilde{\phi} (t_0,-\mu)= \tilde{\phi} (t_0,\mu)$ and $ -\sigma(t_0, \mu) =  \sigma(t_0, -\mu)$.
\end{proof}

\begin{prop}\label{p tildephi and phi are self determined}
The identity $$\tilde{\phi} (t_0, \mu) = \tilde{\phi} (t_0,1) = \phi (t_0)$$ holds for all $t_0\in X_0$ and $\mu \in \mathbb{T}$.
\end{prop}

\begin{proof} The last equality follows from \eqref{eq definition tildephi sigma} and Lemma \ref{l faces maximal non-overlapping}.\smallskip
	
Let us observe that $\tilde{\phi} (t_0, \mu), \phi (t_0)\in Y_0,$ and the latter is a maximal non-overlapping set of $Y$. Thus, if  $\tilde{\phi} (t_0, \mu)\neq \phi (t_0)$, we can find two open disjoint neighbourhoods of these two points, and hence by Remark \ref{r Urysohn's type lemma} there exist two orthogonal or disjoint functions $\tilde{a}\in F^{Y}_{\tilde{\phi} (t_0,\mu),\sigma(t_0, \mu)}$ and $\tilde{b}\in F^{Y}_{{\phi} (t_0),1}$, in particular $\|\tilde{a}\pm \tilde{b}\| =1$. It follows from the defining properties of $\tilde{\phi}$ and $\phi$ that $a= \Delta^{-1} (\tilde{a})\in F^{X}_{t_0, \mu}$ and $b= \Delta^{-1} (\tilde{b})\in F^{X}_{t_0, 1}$.\smallskip

If $\Re(\mu) \leq 0$ we have $$|1- \mu| = |b(t_0) - a(t_0) | \leq \|a-b\| = \|\Delta(a)- \Delta(b)\| = \|\tilde{a}- \tilde{b}\| =1,$$ which is impossible.\smallskip

If $\Re(\mu) > 0$, having in mind that, by Lemma \ref{l sigma is odd and tildephi is even},  $\tilde{\phi} (t_0,\mu) = \tilde{\phi} (t_0,-\mu)$ with $\Re(-\mu) <0$ and $\tilde{\phi} (t_0,\mu) = \tilde{\phi} (t_0,-\mu) \neq \phi (t_0),$ which is impossible by the conclusion of the previous case.
\end{proof}

From now on we shall only work with the bijection $\phi : X_0\to Y_0$ and its inverse (associated to $\Delta^{-1}$, see Lemma \ref{Y0 is a maximal non-overlapping}). For each $t_0\in X_0$, the mapping $\sigma(t_0, \cdot) = \sigma^{\Delta}(t_0, \cdot) : \mathbb{T} \to \mathbb{T}$ is well-defined. Furthermore, by applying the same arguments to $\Delta^{-1}$, we prove the existence of a new mapping $\sigma^{\Delta^{-1}}(\cdot, \cdot) : Y_0\times \mathbb{T} \to \mathbb{T}$ satisfying the appropriate properties described in the comments after Lemma \ref{Y0 is a maximal non-overlapping}.

\begin{lem}\label{l sigmaDelta-1 inverse of sigma delta} For each $t_0\in X_0$ the mappings  $\sigma^{\Delta}(t_0, \cdot), \sigma^{\Delta^{-1}}(\phi(t_0), \cdot) : \mathbb{T}\to \mathbb{T}$ are bijective and $\sigma^{\Delta^{-1}}(\phi(t_0), \cdot)$ is the inverse of  $\sigma^{\Delta}(t_0, \cdot)$.
\end{lem}

\begin{proof} Fix an arbitrary $t_0\in X_0$ and $\mu\in \mathbb{T}$. The conclusion follows almost straightforwardly from the identities $$\begin{aligned}F^{X}_{t_0, \mu} &=\Delta^{-1} \left( F^{Y}_{\phi(t_0), \sigma^{\Delta} (t_0, \mu)}\right) = F^{X}_{\phi^{-1}\phi (t_0), \sigma^{\Delta^{-1}} (\phi(t_0), \sigma^{\Delta} (t_0,\mu))}  \\
		&= F^{X}_{t_0, \sigma^{\Delta^{-1}} (\phi(t_0), \sigma^{\Delta} (t_0,\mu))}
\end{aligned}  $$

\end{proof}

We can argue as in Lemma \ref{lem15} to deduce that $\sigma^{\Delta}(t_0, \cdot)$ is an isometric mapping for each $t_0\in X_0$.

\begin{lem}\label{l sigmat0 dot is an isometry} For each $t_0\in X_0$, the mappings $\sigma^{\Delta}(t_0, \cdot), \sigma^{\Delta^{-1}}(\phi(t_0), \cdot) : \mathbb{T}\to \mathbb{T}$ are surjective isometries.
\end{lem}

\begin{proof} Fix $\mu_1,\mu_2$ in $\mathbb{T}$. Let us fix an element $a\in F_{t_0,1}$. Since, by definition, $\Delta (F^X_{t_0, \mu_j}) = F^Y_{\phi(t_0), \sigma(t_0,\mu_j)}$ it can be easily seen that
$$ \begin{aligned}
	|\sigma(t_0,\mu_1) &-\sigma(t_0,\mu_2)| = |\Delta(\mu_1 a)(\phi(t_0)) - \Delta(\mu_2 a)(\phi(t_0))| \\ &\leq \|\Delta(\mu_1 a) -\Delta(\mu_2 a) \| = \|(\mu_1-\mu _2) a\| = |\mu_1-\mu_2|.
\end{aligned}$$ This proves that $\sigma(t_0,\cdot) =\sigma^{\Delta}(t_0,\cdot)$ is contractive. Replacing $\Delta$ with $\Delta^{-1}$, we deduce that  $\sigma^{\Delta^{-1}}(\phi(t_0), \cdot)$ is contractive too. Since $\sigma^{\Delta^{-1}}(\phi(t_0), \cdot)$ is the inverse of $\sigma^{\Delta}(t_0,\cdot)$ (cf. Lemma \ref{l sigmaDelta-1 inverse of sigma delta}), we can conclude that $\sigma^{\Delta}(t_0, \cdot)$ and $\sigma^{\Delta^{-1}}(\phi(t_0), \cdot)$ are isometries.
\end{proof}

It follows from the previous lemma that for each $t_0\in X_0$, $\sigma^{\Delta}(t_0, \cdot): \mathbb{T}\to \mathbb{T}$ is a surjective isometry. By some of the results commented in the introduction, for example, by the solution to Tingley's problem for $\mathbb{T} = S(\mathbb{C})$ \cite{Ding2002}, we deduce that \begin{equation}\label{eq sigma t0 is identity or conjugation} \sigma^{\Delta}(t_0, \mu ) = \sigma^{\Delta}(t_0, 1 ) \ \mu, \hbox{ or } \sigma^{\Delta}(t_0, \mu ) = \sigma^{\Delta}(t_0, 1) \ \overline{\mu} \ \ (\forall \mu \in \mathbb{T}).
\end{equation} The just stated property determines a partition $X_0 = X_0^{+}\cup X_0^{-}$ with respect to  the subsets
$$X_0^{+} = \{t_0\in X_0 : \sigma^{\Delta}(t_0, \mu ) = \sigma^{\Delta}(t_0, 1 ) \ \mu, \forall \mu\in \mathbb{T}\},$$ and
$$X_0^{-} = \{t_0\in X_0 : \sigma^{\Delta}(t_0, \mu ) = \sigma^{\Delta}(t_0, 1 ) \ \overline{\mu}, \forall \mu\in \mathbb{T}\}.$$\smallskip

The continuous triple functional calculus explained before Lemma \ref{l functional calculus on faces} is now applied in our next technical result.

\begin{lem}\label{l variation of the functional calculus}
Let us fix $t_0\in X_0$ and $a\in S(C_0^{\mathbb{T}}(X))$ with $|a(t_0)| <1$. Set $\lambda = \frac{a(t_0)}{|a(t_0)|}$ if $a(t_0)\neq 0$ and $\lambda =1$ otherwise. Then for each $\varepsilon>0$ there exist $b_{\varepsilon}\in F^{X}_{t_0, 1}$  and $a_{\varepsilon}\in S(C_0^{\mathbb{T}}(X))$ satisfying $$ r a_{\varepsilon} + (1-r |a(t_0)|) \lambda b_{\varepsilon} \in \lambda F^{X}_{t_0,1} = F^{X}_{t_0,\lambda}, \hbox{ for all }  0<r <1,$$ $a_{\varepsilon}(t_0) = a(t_0)$ and $\|a-a_{\varepsilon}\|<\varepsilon$.
\end{lem}

\begin{proof} The case for $a(t_0)=0$ is easier. In such a case, by Remark \ref{r Urysohn's type lemma} and a standard argument, for each $\varepsilon$ there exists $a_{\varepsilon}\in S(C_0^{\mathbb{T}}(X))$ and $b_{\varepsilon}\in F^{X}_{t_0, 1}$ which are orthogonal or disjoint and $\|a-a_{\varepsilon}\|<\varepsilon$. Then clearly, $\|r a_{\varepsilon} + b_{\varepsilon}\| = \max\{r, \|b_{\varepsilon}\|\} =1 =  (r a_{\varepsilon} + b_{\varepsilon}) (t_0)$.\smallskip

Suppose next that $0<|a(t_0)|<1$, and fix a positive $\varepsilon$ such that $|a(t_0)| + \varepsilon <1$. Let us find a $\mathbb{T}$-invariant open neighbourhood $W_{\varepsilon}$ of $t_0$ contained in a 	 $\mathbb{T}$-invariant compact subset such that $|a(s)| < |a(t_0)|+\varepsilon/2,$ for all $s\in W_{\varepsilon}$. Let us find, via Remark \ref{r Urysohn's type lemma} a function $b_{\varepsilon}\in F_{t_0,1}^{X}$ such that $b_{\varepsilon}|_{X\backslash W_{\varepsilon}} \equiv 0$.\smallskip

Let us consider the following $h_{\varepsilon}\in C_0^{\mathbb{T}}(\mathcal{B}_{\mathbb{C}})$ whose values on $[0,1]$ are the following:
 $$
h_{\varepsilon}(s)= \left\{
\begin{array}{ll}
	s, &     0 \leq s \leq |a(t_0)|,\\
	\hbox{affine}, & |a(t_0)| \leq s \leq |a(t_0)|+\varepsilon/2,\\
	|a(t_0)|-\varepsilon/2, &    s= |a(t_0)|+\varepsilon/2, \\
	\hbox{affine}, & |a(t_0)| +\varepsilon/2 \leq s \leq |a(t_0)|+\varepsilon,\\
	s, &     |a(t_0)|+ \varepsilon \leq s \leq 1.    \end{array}%
\right.$$ Let $\iota: \mathcal{B}_{\mathbb{C}} \hookrightarrow \mathbb{C}$ denote the inclusion mapping --note that $\iota\in C_0^{\mathbb{T}} (\mathcal{B}_{\mathbb{C}})$. Set $a_{\varepsilon}:= (h_{\varepsilon})_t (a)$. Since, $\|h_{\varepsilon} - \iota\|=\varepsilon$, it follows that $$\|a_{\varepsilon} - a\|=\| (h_{\varepsilon})_t (a) - \iota(a)\| 
 \leq \varepsilon.$$

Clearly $a_{\varepsilon} (t_0) = a(t_0)$. For $s\in X\backslash W_{\varepsilon}$ we have
$$|(r a_{\varepsilon} + (1-r |a(t_0)|) \lambda b_{\varepsilon}) (s)| = |r a_{\varepsilon}  (s)|\leq r\leq 1.$$

For $s\in W_{\varepsilon}$, we have $|a(s)| < |a(t_0)|+\varepsilon/2,$ and hence $$a_{\varepsilon}(s) = h_{\varepsilon} (a(s)) = h_{\varepsilon} (e^{i \alpha_{s}} |a(s)|) = e^{i \alpha_{s}} h_{\varepsilon} ( |a(s)|) \in e^{i \alpha_{s}} \ [0,|a(t_0)|].$$ Therefore,
$$ |(r a_{\varepsilon} + (1-r |a(t_0)|) \lambda b_{\varepsilon}) (s)| \leq r |a(t_0)| + 1-r |a(t_0)| =1.$$

Finally the identity
$$\begin{aligned}  (r a_{\varepsilon} + (1-r |a(t_0)|) \lambda b_{\varepsilon}) (t_0) & = r a_{\varepsilon} (t_0) + (1-r |a(t_0)|) \frac{a_{\varepsilon} (t_0)}{|a_{\varepsilon} (t_0)|} \\
&=  \frac{a_{\varepsilon} (t_0)}{|a_{\varepsilon} (t_0)|} =  \frac{a (t_0)}{|a (t_0)|} = \lambda,
\end{aligned}$$ proves that  $ r a_{\varepsilon} + (1-r |a(t_0)|) \lambda b_{\varepsilon} \in  F^{X}_{t_0,\lambda},$ as desired.
\end{proof}	

In the next proposition we shall determine the point-evaluations of elements in the image of $\Delta$ at points of the form $\phi(t_0)$.

\begin{prop}\label{p evaluations of images at phi t0} For each $t_0\in X_0$  and each $a\in S(C_0^{\mathbb{T}}(X))$ we have $$\Delta(a) (\phi(t_0)) = \sigma^{\Delta}\left(t_0, \frac{a (t_0)}{|a (t_0)|}\right) |a(t_0)| =   \left\{
	\begin{array}{ll}
	\sigma^{\Delta}\left(t_0, 1\right)	a(t_0), &    \hbox{ if } t_0\in X_0^+,\\
	\ & \ \\
	\sigma^{\Delta}\left(t_0, 1\right)	\overline{a(t_0)}, &   \hbox{ if } t_0\in X_0^-,    \end{array}%
	\right.$$ where $X_0^+$ and $X_0^-$ are the subsets of $X_0$ introduced just before
Lemma~\ref{l variation of the functional calculus}.
\end{prop}	

\begin{proof} Let us fix $t_0\in X_0$ and $a\in S(C_0^{\mathbb{T}}(X))$. The case $a(t_0)=0$ follows from Lemma \ref{l functions vanishing at phit0}. If $|a(t_0)|=1$ we have $a\in F^{X}_{t_0,a(t_0)}$ and thus
$$\Delta(a)(\phi(t_0)) = \sigma^{\Delta}(t_0, a(t_0)) =   \left\{
\begin{array}{ll}
	\sigma^{\Delta}\left(t_0, 1\right)	a(t_0), &    \hbox{ if } t_0\in X_0^+,\\
	\ & \ \\
	\sigma^{\Delta}\left(t_0, 1\right)	\overline{a(t_0)}, &   \hbox{ if } t_0\in X_0^-    \end{array}%
\right.$$ (cf. \eqref{eq definition tildephi sigma}, Proposition \ref{p tildephi and phi are self determined} and \eqref{eq sigma t0 is identity or conjugation}). We can therefore assume that $1> |a(t_0)| > 0$. Set $\lambda = \frac{a(t_0)}{|a(t_0)|}$.\smallskip

We shall first show that \begin{equation}\label{eq modules of evaluations coincide} | \Delta(a) (\phi(t_0))| =|a(t_0)|.
\end{equation}	
	
By Lemma \ref{l variation of the functional calculus} for each $\varepsilon>0$ there exist $b_{\varepsilon}\in F^{X}_{t_0, 1}$  and $a_{\varepsilon}\in S(C_0^{\mathbb{T}}(X))$ satisfying $$ c_{r,\varepsilon} = r a_{\varepsilon} + (1-r |a(t_0)|) \lambda b_{\varepsilon} \in \lambda F^{X}_{t_0,1} = F^{X}_{t_0,\lambda}, \hbox{ for all } 0<r <1,$$ $a_{\varepsilon}(t_0) = a(t_0)$ and $\|a-a_{\varepsilon}\|<\varepsilon$. In particular,
$$\Delta (c_{r,\varepsilon} ) (\phi(t_0)) = \sigma^{\Delta} (t_0, \lambda),$$ and by definition,
$$\|c_{r,\varepsilon}-a_{\varepsilon} \| \leq (1-r) + 1- r |a(t_0)| = 2 -r -r |a(t_0)|.$$ On the other hand,
\begin{equation}\label{eq long equality on 0922}
\begin{aligned}
	1-|\Delta(a_{\varepsilon})(\phi(t_0))| &= |\sigma^{\Delta} (t_0, \lambda)| -|\Delta(a_{\varepsilon})(\phi(t_0))| \\
	&\leq |\sigma^{\Delta} (t_0, \lambda) - \Delta(a_{\varepsilon})(\phi(t_0))| \\
	& = |\Delta (c_{r,\varepsilon} ) (\phi(t_0)) - \Delta(a_{\varepsilon})(\phi(t_0))| \\
	&\leq  \|\Delta (c_{r,\varepsilon} ) - \Delta(a_{\varepsilon})\| =\|c_{r,\varepsilon} - a_{\varepsilon}\|\\
	& \leq 2 -r -r |a(t_0)|,
\end{aligned}
\end{equation} witnessing that $r +r |a(t_0)|-1 \leq |\Delta(a_{\varepsilon})(\phi(t_0))|$ ($0< r<1$). Letting $r\to 1$ we get $|a(t_0)| \leq |\Delta(a_{\varepsilon})(\phi(t_0))|$. Now, letting $\varepsilon$ tend to zero, it follows from the continuity of $\Delta$ that $$|a(t_0)| \leq |\Delta(a)(\phi(t_0))|.$$

Applying the same argument to $\Delta^{-1},$ $\phi^{-1},$ $\Delta (a)$ and $\phi(t_0)$ in the roles of $\Delta$, $\phi,$ $a$ and $t_0$ we get $$ |\Delta(a) (\phi(t_0))| \leq |\Delta^{-1} \Delta(a)(\phi^{-1}\phi(t_0))| = |a(t_0)|,$$ which concludes the proof of \eqref{eq modules of evaluations coincide}.\smallskip

If we take limits $r\to 1$ and $\varepsilon\to 0$ in the inequalities given by the second and last lines of \eqref{eq long equality on 0922} we arrive to \begin{equation}\label{eq derived from limits in long equality} 
	\left|\sigma^{\Delta} \left(t_0, \frac{a(t_0)}{|a(t_0)|}\right) -\Delta(a)(\phi(t_0))\right| \leq 1-|a(t_0)|.
\end{equation} Consequently, by \eqref{eq modules of evaluations coincide}, we get
$$\begin{aligned}
1 = \left|\sigma^{\Delta} \left(t_0, \frac{a(t_0)}{|a(t_0)|}\right) \right| &\leq \left|\sigma^{\Delta} \left(t_0, \frac{a(t_0)}{|a(t_0)|}\right) -\Delta(a)(\phi(t_0))\right| +\left| \Delta(a)(\phi(t_0))\right|	\\
& \leq 1-|a(t_0)| + |a(t_0)| =1.
\end{aligned}	
$$ It then follows that the equality holds in a triangular inequality, so there exists a positive $\delta>0$ such that $ \delta \sigma^{\Delta} \left(t_0, \frac{a(t_0)}{|a(t_0)|}\right) = \Delta(a)(\phi(t_0))$, and in particular $\delta = |\Delta(a)(\phi(t_0))| = |a(t_0)|$.  We have therefore proved that
$$  \Delta(a)(\phi(t_0)) = \sigma^{\Delta} \left(t_0, \frac{a(t_0)}{|a(t_0)|}\right) \ |a(t_0)|.$$ The rest is clear from \eqref{eq sigma t0 is identity or conjugation} and the subsequent comments.
\end{proof}

For each $t_0$ in a principal $\mathbb{T}$-bundle $X$, we shall write $\overline{\delta_{t_0}}$ for the conjugate linear functional on $C_0^{\mathbb{T}}(X)$ defined by $\overline{\delta_{t_0}} (a) = \overline{a(t_0)}$ ($a\in C_0^{\mathbb{T}}(X)$).

\begin{proof}[Proof of Theorem \ref{t Tingleys problem for commutative JBstar triples}] Let $X_0$ and $Y_0$ denote the maximal non-overlapping\hyphenation{non-over-lapping} subsets employed in the previous arguments. Let $\phi : X_0\to Y_0$ be the bijection presented in Lemmata \ref{l faces maximal non-overlapping} and \ref{Y0 is a maximal non-overlapping}. As we have already commented, the sets $\{\delta_{t_0} : t_0\in X_0\}$ and $\{\delta_{\phi(t_0)} : t_0\in X_0\}$ are norming in $C_0^{\mathbb{T}}(X)^*$ and $C_0^{\mathbb{T}}(Y)^*$, respectively. The same property holds for the set $$\begin{aligned}
\{\sigma^{\Delta}\left(t_0, 1\right)  \delta_{\phi(t_0)} : t_0\in X_0^+\} &\cup 	 \{\sigma^{\Delta}\left(t_0, 1\right)  \overline{\delta_{\phi(t_0)}} : t_0\in X_0^-\}  \\	= \{\delta_{\sigma^{\Delta}\left(t_0, 1\right) \phi(t_0)} : t_0\in X_0^+\} &\cup \{\overline{\delta_{\overline{\sigma^{\Delta}\left(t_0, 1\right)} \phi(t_0)}} : t_0\in X_0^-\}.		
	\end{aligned}$$

By Proposition \ref{p evaluations of images at phi t0} for each $a\in S(C_0^{\mathbb{T}}(X))$ we have \begin{equation}\label{eq evaluations at functionals in the families} \delta_{\phi(t_0) }\left( \Delta(a)\right) =    \left\{
	\begin{array}{ll}
		\delta_{\sigma^{\Delta}\left(t_0, 1\right) \phi(t_0)} (a), &    \hbox{ if } t_0\in X_0^+,\\
		\ & \ \\
		\overline{\delta_{\overline{\sigma^{\Delta}\left(t_0, 1\right)} \phi(t_0)}} (a), &   \hbox{ if } t_0\in X_0^-.    \end{array}%
	\right.
\end{equation}

Lemma 6 in \cite{MoriOza2018} (see also \cite[Lemma 2.1]{FangWang06}) implies that $\Delta$ admits an extension to a surjective real linear isometry. We can alternatively follow a similar argument to that in the proof of Theorem \ref{thm1}, and employ the identity in \eqref{eq evaluations at functionals in the families} to prove that the positive homogeneous extension of $\Delta$ is additive, and hence a real linear extension of $\Delta$. The final conclusions are straightforward consequences of Lemma \ref{l form of surjective real linear isometries between abelian JBstar triples}.
\end{proof}

\smallskip\smallskip

\textbf{Acknowledgements} First author partially supported by EPSRC (UK) project ``Jordan Algebras, Finsler Geometry and Dynamics'' ref. no. EP/R044228/1 and by the Spanish Ministry of Science, Innovation and Universities (MICINN) and European Regional Development Fund project no. PGC2018-093332-B-I00, Junta de Andaluc\'{\i}a grants FQM375 and A-FQM-242-UGR18. Third  author partially supported by JSPS KAKENHI (Japan) Grant Number JP 20K03650. Fourth author partially supported by MCIN / AEI / 10. 13039 / 501100011033 / FEDER ``Una manera de hacer Europa'' project no. PGC2018-093332-B-I00, Junta de Andaluc\'{\i}a grants FQM375, A-FQM-242-UGR18 and PY20$\underline{\ }$ 00255, and by the IMAG--Mar{\'i}a de Maeztu grant CEX2020-001105-M / AEI / 10.13039 / 501100011033.


\end{document}